\documentclass[a4paper,10pt]{article}

\usepackage[utf8]{inputenc}
\usepackage{amsmath,amsthm,amssymb}
\usepackage{hyperref}
\usepackage{multirow}
\usepackage{dsfont}
\usepackage{mathrsfs}
\usepackage{geometry}

\newtheorem{theorem}{Theorem}
\newtheorem{lemma}{Lemma}
\newtheorem{corollary}{Corollary}
\newtheorem{definition}{Definition}
\newtheorem{remark}{Remark}
\newtheorem{proposition}{Proposition}

\title{Optimal Distributed and Tangential Boundary Control for the Unsteady Stochastic Stokes Equations}
\date{August 22, 2018}
\author{Peter Benner \thanks{Max Planck Institute for Dynamics of Complex Technical Systems, Sandtorstraße 1, 39106 Magdeburg, Germany (benner@mpi-magdeburg.mpg.de).} \qquad
Christoph Trautwein \thanks{Max Planck Institute for Dynamics of Complex Technical Systems, Sandtorstraße 1, 39106 Magdeburg, Germany (trautwein@mpi-magdeburg.mpg.de).}}

\begin{document}

\maketitle

\begin{abstract}
 We consider a control problem constrained by the unsteady stochastic Stokes equations with nonhomogeneous boundary conditions in connected and bounded domains.
 In this paper, controls are defined inside the domain as well as on the boundary.
 Using a stochastic maximum principle, we derive necessary and sufficient optimality conditions such that explicit formulas for the optimal controls are derived.
 As a consequence, we are able to control the stochastic Stokes equations using distributed controls as well as boundary controls in a desired way.
\end{abstract}

\section{Introduction}

In this paper, we consider a linear quadratic control problem for the unsteady stochastic Stokes equations with linear multiplicative noise.
Here, controls appear as distributed controls inside the domain as well as tangential controls on the boundary.
Concerning fluid dynamics, noise enters the system due to structural vibration and other environmental effects, see \cite{aitd} and the references therein.
The aim is to find controls such that the velocity field is as close as possible to a given desired velocity field.

In the last decades, optimal control problems constrained by the Stokes equations have been studied extensively.
Simultaneous distributed and boundary controls can be found in \cite{socf}.
In \cite{ocotse,ocot}, discretization schemes for control problems are considered.
For stochastic distributed controls, we refer to \cite{mawr}.
In \cite{ocfm}, an optimal control problem for the Stokes equations is presented, where the viscosity satisfies a transport equation.
A control problem motivated by Stokes flow in an artificial heart is considered in \cite{wpao}.
We extend this setting by allowing additional noise terms arising from random environmental effects.
We overcome this problem by decomposing the external force into a control term and a noise term.
Moreover, control problems are mainly considered for the case of distributed controls.
Therefore, we include nonhomogeneous Dirichlet boundary conditions to involve tangential boundary controls.

Using stochastic processes, one can model structural vibration and other environmental effects affecting flow fields.
This leads us immediately to the formulation of a stochastic partial differential equation, which belongs to the modern research areas of infinite dimensional stochastic analysis.
Such equations can be interpreted as stochastic evolution equations and the solutions are defined in a generalized sense.
There exist different approaches on how to deal with these solutions.
In \cite{seiid,edos,sdei,acco}, the concept of weak solutions is introduced, where the construction in mainly based on inner products.
Using Gelfand triples, another approach is given by variational solutions, see \cite{acco,ses}.
For problems containing a linear operator as the generator of a semigroup on a Hilbert space, one can use mild solutions, see \cite{seiid,edos,sdei}.
Mild solutions are considered as solutions to integral equations of Itô-Volterra type containing a stochastic convolution.
All of these concepts are based on a given probability space and they are called (probabilistic) strong solutions.
Solutions constructing the probability space are called (probabilistic) weak solutions or martingale solutions, see \cite{seiid,edos}.

In this paper, we use the theory of mild solutions in order to cover especially the nonhomogeneous boundary conditions.
The construction of the solution is mainly based on an approach for the deterministic Stokes equations, see \cite{sans}.
Although this approach is applicable for a broad class of boundary conditions, we restrict to the case of tangential boundary conditions.
Therefor, we can reformulate the Stokes equations as an evolution equation in a suitable Hilbert space.
Since we assume that the external force can be decomposed into a control term and a noise term, we obtain immediately a linear stochastic partial differential equation with distributed and Dirichlet boundary controls.
We prove the existence and uniqueness of a mild solution being square integrable with respect to the time variable.
In order to get a well defined solution, we need the definition of stochastic integrals with respect to adapted processes, see \cite{sdei}.

The control problem considered in this paper is formulated as a tracking problem motivated by \cite{apfo,mawr,epfa,fbso,ocot}.
We derive a stochastic maximum principle to obtain first order optimality conditions, which are necessary and sufficient.
To utilize these optimality conditions, a duality principle is required.
In general, a duality principle gives a relation between forward and backward stochastic partial differential equations using an Itô product formula, which is not applicable for mild solutions.
Hence, we approximate the mild solutions by strong solutions using an approach based on the resolvent operator, see \cite{yaos,soss}.
As a consequence, we obtain the duality principle for the approximating strong solutions and due to convergence results, the duality principle holds also for the mild solutions.
Based on the optimality conditions and the duality principle, we deduce formulas the optimal distributed control and the optimal boundary control have to satisfy.

The main contribution of this paper is to provide a mild solution to the stochastic Stokes equations with nonhomogeneous tangential boundary conditions.
Moreover, we solve a control problem using a stochastic maximum principle such that optimal distributed controls and optimal boundary controls are derived.

The paper is organized as follows.
In Section 2, we introduce common spaces and operators concerning the Stokes equations.
Moreover, we discuss the deterministic Stokes equations with nonhomogeneous boundary conditions and we give an introduction to stochastic integrals with respect to adapted processes.
In Section 3, we provide an existence and uniqueness result for the stochastic Stokes equations with nonhomogeneous boundary conditions.
Section 4 addresses the control problem.
We derive optimality conditions and a duality principle such that formulas for the optimal distributed control as well as the optimal boundary control are derived.

\section{Preliminaries}\label{sec:preliminaries}

\subsection{Functional Analysis Background}\label{sec:functionalbackground}

Throughout the paper, let $\mathcal D \subset \mathbb{R}^n$, $n \geq 2$, be a connected and bounded domain with $C^2$ boundary $\partial \mathcal D$.
For $s \geq 0$, let $H^s(\mathcal D)$ denote the usual Sobolev space and for $s \geq \frac{1}{2}$, let $H_0^s(\mathcal D) = \left\{y \in H^s(\mathcal D) \colon y = 0 \text{ on } \partial \mathcal D \right\}$.
We introduce the following common spaces:
\begin{align*}
 H &= \text{Completion of } \{y \in (C^\infty_0(\mathcal D))^n \colon \text{div }y = 0 \text{ in } \mathcal D\} \text{ in } (L^2(\mathcal D))^n \\
 &= \left\{ y \in (L^2(\mathcal D))^n \colon \text{div }y = 0 \text{ in } \mathcal D, y \cdot \eta = 0 \text{ on } \partial \mathcal D\right\}, \\
 V &= \text{Completion of } \{y \in (C^\infty_0(\mathcal D))^n \colon \text{div }y = 0 \text{ in } \mathcal D\} \text{ in } \left(H^1(\mathcal D)\right)^n \\
 &= \left\{ y \in \left(H^1_0(\mathcal D)\right)^n \colon \text{div }y = 0 \text{ in } \mathcal D\right\},
\end{align*}
where $\eta$ denotes the unit outward normal to $\partial \mathcal D$.
The space $H$ equipped with the inner product 
\begin{equation*}
 \langle y,z \rangle_H = \langle y,z \rangle_{(L^2(\mathcal D))^n} = \int\limits_{\mathcal D} \sum_{i =1}^n y_i(x) z_i(x) \, dx
\end{equation*}
for every $y = (y_1,...,y_n), z = (z_1,...,z_n) \in H$ becomes a Hilbert space.
For all $x = (x_1,...,x_n) \in \mathcal D$, we denote $D^j = \frac{\partial^{|j|}}{\partial x_1^{j_1} \cdot \cdot \cdot \partial x_n^{j_n}}$ with $|j| = \sum_{i=1}^n j_i$.
We set $D^j y = (D^j y_1,...,D^j y_n)$ for every $y= (y_1,...,y_n) \in V$ and $|j| \leq 1$.
Then the space $V$ equipped with the inner product
\begin{equation*}
 \langle y,z \rangle_V = \sum_{|j| \leq 1} \langle D^j y , D^j z\rangle_{(L^2(\mathcal D))^n}
\end{equation*}
for every $y,z \in V$ becomes a Hilbert space.
The norm in $H$ and $V$ is denoted by $\|\cdot\|_H$ and $\|\cdot\|_V$, respectively.
We get the orthogonal Helmholtz decomposition
\begin{equation*}
 (L^2(\mathcal D))^n = H \oplus \{\nabla y : y \in H^1(\mathcal D)\},
\end{equation*}
where $\oplus$ denotes the direct sum.
Then there exists an orthogonal projection $\Pi \colon (L^2(\mathcal D))^n \rightarrow H$, see \cite{alto}.
Next, we define the Stokes Operator $A \colon D(A) \subset H \rightarrow H$ by $A y = - \Pi \Delta y$ for every $y \in D(A)$, where $D(A) = \left(H^2(\mathcal D)\right)^n \cap V$.
The Stokes operator $A$ is positive, self adjoint and has a bounded inverse.
Moreover, the operator $-A$ is the infinitesimal generator of an analytic semigroup $(e^{-At})_{t \geq 0}$ such that $\left\| e^{-At} \right\|_{\mathcal L(H)} \leq 1$ for all $t \geq 0$. 
For more details, see \cite{ofpo,aots,silo,c0sa}.
Hence, we can introduce fractional powers of the Stokes operator, see \cite{solo,teon,c0sa}.
For $\alpha > 0$, we define
\begin{equation}\label{negativefractionalequation}
 A^{-\alpha} = \frac{1}{\Gamma(\alpha)} \int\limits_0^\infty t^{\alpha-1} e^{-A t} dt,
\end{equation}
where $\Gamma (\cdot)$ denotes the gamma function.
The operator $A^{-\alpha}$ is linear, bounded and injective in $H$.
Hence, we define for all $\alpha > 0$
\begin{equation*}
 A^\alpha = \left(A^{-\alpha}\right)^{-1}.
\end{equation*}
Moreover, we set $A^0 = I$, where $I$ is the identity operator in $H$.
For $\alpha > 0$, the operator $A^\alpha$ is linear and closed in $H$ with dense domain $D(A^\alpha) = R(A^{-\alpha})$, where $R(A^{-\alpha})$ denotes the range of $A^{-\alpha}$.
Next, we provide some useful properties of fractional powers of the Stokes operator.

\begin{lemma}[cf. Section 2.6,{\cite{solo}}]\label{fractional}
 Let $A \colon D(A) \subset H \rightarrow H$ be the Stokes operator.
 Then
 \begin{itemize}
  \item[(i)] for $\alpha, \beta \in \mathbb R$, we have $A^{\alpha+\beta} y = A^\alpha A^\beta y$ for every $y \in D(A^\gamma)$, where $\gamma = \max \{\alpha,\beta,\alpha+\beta\}$,
  \item[(ii)] $e^{-At} \colon H \rightarrow D(A^\alpha)$ for all $t>0$ and $\alpha \geq 0$,
  \item[(iii)] we have $A^\alpha e^{-At} y = e^{-At} A^\alpha y$ for every $y \in D(A^\alpha)$ with $\alpha \in \mathbb R$,
  \item[(iv)] the operator $A^\alpha e^{-At}$ is bounded for all $t>0$ and there exist constants $M_\alpha,\theta > 0$ such that
  \begin{equation*}
   \left\| A^{\alpha}e^{-At} \right\|_{\mathcal L(H)} \leq M_\alpha t^{-\alpha}e^{-\theta t},
  \end{equation*}
  \item[(v)] $0 \leq \beta \leq \alpha \leq 1$ implies $D(A^\alpha) \subset D(A^\beta)$ and there exists a constant $C > 0$ such that for every $y \in D(A^\alpha)$
  \begin{equation*}
   \left\| A^\beta y \right\|_H \leq C \left\| A^{\alpha} y\right\|_H.
  \end{equation*}
 \end{itemize}
\end{lemma}

As a consequence of the previous lemma, we obtain that the space $D(A^\alpha)$ for all $\alpha \geq 0$ equipped with the inner product 
\begin{equation*}
 \langle y,z\rangle_{D(A^\alpha)} = \langle A^\alpha y,A^\alpha z\rangle_H
\end{equation*}
for every $y,z \in D(A^\alpha)$ becomes a Hilbert space.
Furthermore, we get the following result.

\begin{lemma}\label{fractionalselfadjoint}
 Let $A \colon D(A) \subset H \rightarrow H$ be the Stokes operator.
 Then the operator $A^\alpha$ is self adjoint for all $\alpha \in \mathbb R$.
\end{lemma}

\begin{proof}
 First, we show the claim for negative exponents.
 Recall the Stokes operator $-A$ is self adjoint.
 Hence, the semigroup $(e^{-At})_{t \geq 0}$ is self adjoint as well.
 By equation (\ref{negativefractionalequation}), we get for every $y,z \in \mathcal H$ and all $\alpha>0$
 \begin{align}\label{negativeeq}
  \left \langle A^{-\alpha} y, z\right \rangle_{\mathcal H}
  &= \left \langle \frac{1}{\Gamma(\alpha)} \int\limits_0^\infty t^{\alpha-1} S(t) y dt , z\right \rangle_{\mathcal H} = \frac{1}{\Gamma(\alpha)} \int\limits_0^\infty t^{\alpha-1} \left \langle S(t) y , z\right \rangle_{\mathcal H} dt \nonumber \\
  &= \left \langle y, \frac{1}{\Gamma(\alpha)} \int\limits_0^\infty t^{\alpha-1} S(t) z dt \right \rangle_{\mathcal H} = \left \langle y, A^{-\alpha} z\right \rangle_{\mathcal H}.
 \end{align}
 Next, we show the claim for positive exponents.
 Using Theorem \ref{fractional} (iv) and equation (\ref{negativeeq}), we obtain for every $y,z \in D(A^\alpha)$ and all $\alpha>0$
 \begin{equation*}
  \left \langle A^\alpha y, z\right \rangle_{\mathcal H} = \left \langle A^\alpha y, A^{-\alpha} A^\alpha z\right \rangle_{\mathcal H} = \left \langle A^{-\alpha} A^\alpha y, A^\alpha z\right \rangle_{\mathcal H} = \left \langle y, A^\alpha z\right \rangle_{\mathcal H}.
 \end{equation*}
 For $\alpha = 0$, the claim is obvious.
\end{proof}

Next, we introduce the resolvent operator of $-A$ and we state some of its basic properties.
For more details, see \cite{solo}.
Let $\lambda \in \mathbb R$ be such that $\lambda I + A$ is invertible, i.e. $(\lambda I + A)^{-1}$ is a linear and bounded operator in the space $H$.
Then the operator $R(\lambda;-A) = (\lambda I + A)^{-1}$ is called the resolvent operator.
The operator $R(\lambda;-A)$ maps $H$ into $D(A)$ and using the closed graph theorem, we can conclude that the operator $A R(\lambda;-A) \colon H \rightarrow H$ is linear and bounded.
Moreover, we have the following representation:
\begin{equation} \label{resolventoperator}
 R(\lambda;-A) = \int\limits_0^\infty e^{-\lambda r} e^{-A r} dr.
\end{equation}
For all $\lambda > 0$, we get 
\begin{equation*}
 \|R(\lambda;-A)\|_{\mathcal L (H)} \leq \frac{1}{\lambda}
\end{equation*}
and since the semigroup $(e^{-At})_{t \geq 0}$ is self adjoint, the operator $R(\lambda;-A)$ is self adjoint as well.
Let the operator $R(\lambda) \colon H \rightarrow D(A)$ be defined by $R(\lambda) = \lambda R(\lambda;-A)$.
Hence, we get for all $\lambda > 0$
\begin{equation}\label{resolventinequality}
 \|R(\lambda)\|_{\mathcal L (H)} \leq 1.
\end{equation}
By Lemma \ref{fractional} (iii) and equation (\ref{resolventoperator}), we obtain for every $y \in D(A^\alpha)$ with $\alpha \in \mathbb R$
\begin{equation}\label{resolventcommutes}
 A^\alpha R(\lambda) y = R(\lambda) A^\alpha y.
\end{equation}
Moreover, we have for every $y \in H$
\begin{equation}\label{resolventconvergence}
 \lim\limits_{\lambda \rightarrow \infty} \| R(\lambda) y - y\|_H = 0.
\end{equation}

If the domain $\mathcal D$ is connected and bounded with $C^\infty$ boundary $\partial \mathcal D$, then we can specify the domain of the operator $A^\alpha$ for $\alpha \in (0,1)$ explicitly.
Let $A_D \colon D(A_D) \subset \left(L^2(\mathcal D)\right)^n \rightarrow \left(L^2(\mathcal D)\right)^n$ be the Laplace operator with homogeneous Dirichlet boundary condition defined by $A_D y = -\Delta y$ for all $y \in D(A_D)$.
The domain is given by 
\begin{equation*}
 D(A_D) = \left(H_0^1(\mathcal D)\right)^n \cap \left(H^2(\mathcal D)\right)^n.
\end{equation*}
Then $A_D$ is a positive and self adjoint operator and $-A_D$ is the infinitesimal generator of an analytic semigroup $(e^{-A_D t})_{t \geq 0}$ such that $\left\| e^{-A_D t} \right\|_{\mathcal L(H)} \leq 1$ for all $t \geq 0$.
Hence, we can define fractional powers of the Laplace operator denoted by $A^\alpha_D$ for $\alpha \in \mathbb R$.
We get the following result.

\begin{proposition}[Theorem 1.1,{\cite{ofpo}}]\label{domainfractional1}
 Let the operator $A \colon D(A) \subset H \rightarrow H$ be the Stokes Operator and let the operator $A_D \colon D(A_D) \subset \left(L^2(\mathcal D)\right)^n \rightarrow \left(L^2(\mathcal D)\right)^n$ be the Laplace operator with homogeneous Dirichlet boundary condition.
 Then we have for any $\alpha \in (0,1)$
 \begin{equation*}
  D(A^\alpha) = D(A_D^\alpha) \cap H.
 \end{equation*}
\end{proposition}

The domain of the operator $A_D^\alpha$ can be determined explicitly for $\alpha \in (0,1)$.

\begin{proposition}[cf. Theorem 1,{\cite{ccot}}]\label{domainfractional2}
 Let $A_D \colon D(A_D) \subset \left(L^2(\mathcal D)\right)^n \rightarrow \left(L^2(\mathcal D)\right)^n$ be the Laplace operator with homogeneous Dirichlet boundary condition.
 Then we have
 \begin{itemize}
  \item[(i)]  $D(A_D^\alpha) = \left(H^{2 \alpha} (\mathcal D)\right)^n$ for $\alpha \in \left(0,\frac{1}{4}\right)$,
  \item[(ii)] $D(A_D^{1/4}) \subset \left(H^{1/2} (\mathcal D)\right)^n$,
  \item[(iii)] $D(A_D^\alpha) = \left(H_0^{2 \alpha} (\mathcal D)\right)^n$ for $\alpha \in \left(\frac{1}{4},\frac{3}{4}\right)$,
  \item[(ii)] $D(A_D^{3/4}) \subset \left(H_0^{3/2} (\mathcal D)\right)^n$,
  \item[(v)]  $D(A_D^\alpha) = \left(H_0^{2 \alpha} (\mathcal D)\right)^n$ for $\alpha \in \left(\frac{3}{4},1\right)$.
 \end{itemize}
\end{proposition}

\subsection{The Stokes Equations}\label{sec:stokesequations}

In this section, we consider the deterministic Stokes equations with nonhomogeneous boundary conditions.
Here, we restrict the problem to tangential boundary conditions.
A general formulation can be found in \cite{sans}.

Throughout the paper, let $T>0$.
We introduce the Stokes equations with nonhomogeneous boundary conditions:
\begin{equation}\label{detstokes}
 \left\{
 \begin{aligned}
  \frac{\partial}{\partial t} y(t,x) - \Delta y(t,x) + \nabla p(t,x) &= f(t,x) & &\text{in } (0,T) \times \mathcal D, \\
  \text{div } y(t,x) &= 0 & &\text{in } (0,T) \times \mathcal D, \\
  y(t,x) &= g(t,x) & &\text{on } (0,T) \times \partial \mathcal D, \\
  y(0,x) &= \xi(x) & &\text{in } \mathcal D,
 \end{aligned}
 \right.
\end{equation}
where $y(t,x) \in \mathbb R^n$ denotes the velocity field with initial value $\xi(x) \in \mathbb R^n$, $p(t,x) \in \mathbb R$ describes the pressure of the fluid, and $f(t,x) \in \mathbb R^n$ is the external force.
The boundary condition $g(t,x) \in \mathbb R^n$ is assumed to be tangential, i.e.
\begin{equation*}
 g(t,x) \cdot \eta(x) = 0 \quad \text{on } (0,T) \times \partial \mathcal D,
\end{equation*}
where $\eta$ denotes the unit outward normal to $\partial \mathcal D$.
The goal is to reformulate system (\ref{detstokes}) as an evolution equation.
We define the following spaces for $s \geq 0$:
\begin{align*}
 V^s(\mathcal D) &= \left\{ y \in \left(H^s (\mathcal D)\right)^n \colon \text{div }y = 0 \text{ in } \mathcal D, y \cdot \eta = 0 \text{ on } \partial \mathcal D \right\}, \\
 V^s(\partial \mathcal D) &= \left\{ y \in \left(H^s (\partial \mathcal D)\right)^n \colon y \cdot \eta = 0 \text{ on } \partial \mathcal D \right\}.
\end{align*}
For $s < 0$, the space $V^s(\partial \mathcal D)$ is the dual space of $V^{-s}(\partial \mathcal D)$ with $V^0(\partial \mathcal D)$ as pivot space.
Moreover, let $H^s(\mathcal D)/\mathbb R$ with $s \geq 0$ be the quotient space of $H^s(\mathcal D)$ by $\mathbb R$, i.e. $H^s(\mathcal D)/\mathbb R = \{ y + c \colon y \in H^s(\mathcal D), c \in \mathbb R\}$.
We set $\|y\|_{H^s(\mathcal D) / \mathbb R} = \inf_{c \in \mathbb R} \| y + c\|_{H^s(\mathcal D)}$ for every $y \in H^s(\mathcal D)/\mathbb R$.
The dual space is denoted by $(H^s(\mathcal D)/\mathbb R)'$ with $H^0(\mathcal D)/\mathbb R$ as pivot space.

Next, let us consider the system
\begin{equation} \label{dirichletequation}
 \left\{
 \begin{aligned}
  - \Delta w + \nabla \pi &= 0 \quad \text{and} \quad \text{div }w = 0& &\text{in } \mathcal D, \\
  w &= g& &\text{on } \partial \mathcal D.
 \end{aligned}
 \right.
\end{equation}
We have the following existence and uniqueness results.

\begin{proposition}[cf. Theorem IV.6.1,{\cite{aitt}}]\label{existence1}
 If $g \in V^{3/2}(\partial \mathcal D)$, then there exists a unique solution $(w,\pi) \in V^2(\mathcal D) \times H^1(\mathcal D) / \mathbb R$ of system (\ref{dirichletequation}) and the following estimate holds:
 \begin{equation*}
  \|w\|_{V^2(\mathcal D)} + \|\pi\|_{H^1(\mathcal D) / \mathbb R} \leq C^* \|g\|_{V^{3/2}(\partial \mathcal D)},
 \end{equation*}
 where $C^*>0$ is a constant.
\end{proposition}

\begin{proposition}[cf. \cite{acos,sans}]\label{existence2}
 If $g \in V^{-1/2}(\partial \mathcal D)$, then there exists a unique solution $(w,\pi) \in V^0(\mathcal D) \times \left(H^1(\mathcal D) / \mathbb R \right)'$ of system (\ref{dirichletequation}) and the following estimate holds:
 \begin{equation*}
  \|w\|_{V^0(\mathcal D)} + \|\pi\|_{(H^1(\mathcal D) / \mathbb R)'} \leq C^* \|g\|_{V^{3/2}(\partial \mathcal D)},
 \end{equation*}
 where $C^*>0$ is a constant.
\end{proposition}

We introduce the Dirichlet operators $D$ and $D_p$ defined by
\begin{equation*}
 D g = w \quad \text{and} \quad D_p g = \pi,
\end{equation*}
where $(w,\pi)$ is the solution of system (\ref{dirichletequation}).
We get the following properties of the Dirichlet operators, which is an immediate consequence of Proposition \ref{existence1} and Proposition \ref{existence2}.

\begin{corollary}[cf. Corollary A.1, {\cite{sans}}]\label{dirichletoperatorregularity}
 The operator $D$ is linear and continuous from $V^s(\partial \mathcal D)$ into $V^{s+1/2}(\mathcal D)$ for all $-\frac{1}{2} \leq s \leq \frac{3}{2}$.
 If $-\frac{1}{2} \leq s < \frac{1}{2}$, then the operator $D_p$ is linear and continuous from $V^s(\partial \mathcal D)$ into $\left(H^{1/2-s}(\mathcal D) / \mathbb R \right)'$, and if $\frac{1}{2} \leq s \leq \frac{3}{2}$, then the operator $D_p$ is linear and continuous from $V^s(\partial \mathcal D)$ into $H^{s-1/2}(\mathcal D) / \mathbb R$.
\end{corollary}

As a consequence of Proposition \ref{domainfractional1}, Proposition \ref{domainfractional2} and Corollary \ref{dirichletoperatorregularity}, we get $D \in \mathcal L \left(V^0(\partial \mathcal D);D(A^\beta)\right)$ for $\beta \in \left(0,\frac{1}{4}\right)$.
By the closed graph theorem, we have $A^\beta D \in \mathcal L \left(V^0(\partial \mathcal D);V^0(\mathcal D)\right)$.
Note that $V^0(\mathcal D) = H$.
Furthermore, system (\ref{detstokes}) can be rewritten in the following form:
\begin{equation}\label{abstractstokes}
 \left\{
 \begin{aligned}
  \frac{d}{d t} y(t) &= - A y(t) + A D g(t) + \Pi f(t), \\
  y(0) &= \Pi \xi,
 \end{aligned}
 \right.
\end{equation}
where the operators $A$ and $\Pi$ are introduced in Section \ref{sec:functionalbackground}.
For the sake of simplicity, we assume $f(t),\xi \in H$ for $t \in [0,T]$.
Hence, we obtain a linear evolution equation and the solution is given by
\begin{equation*}
 y(t) = e^{-A t} \xi + \int\limits_0^t A e^{-A (t-s)} D g(s) ds + \int\limits_0^t e^{-A (t-s)} f(s) ds.
\end{equation*}
For more details about linear evolution equations, see \cite{raco}.
The following existence and uniqueness result is stated in \cite{sans} for more general boundary conditions and $f=0$.

\begin{theorem}
 Let $g \in L^2([0,T];V^0(\partial \mathcal D))$ and $f \in L^2([0,T];H)$.
 If $\alpha \in [0,\frac{1}{4})$, then for any $\xi \in D(A^\alpha)$, there exists a unique solution $y \in L^2([0,T];D(A^\alpha))$ of system (\ref{abstractstokes}) and the following estimate holds:
 \begin{equation*}
  \|y\|_{L^2([0,T];D(A^\alpha))} \leq C^* \left(\|\xi\|_{D(A^\alpha)} + \|g\|_{L^2([0,T];V^0(\partial \mathcal D))} + \|f\|_{L^2([0,T];H)}\right),
 \end{equation*}
 where $C^*>0$ is a constant.
\end{theorem}

\subsection{Stochastic Processes and the Stochastic Integral}

In this section, we give a brief introduction to stochastic integrals, where the noise term is defined as a Hilbert space valued Wiener process.
For more details, see \cite{seiid}.

Let $(\Omega,\mathcal{F}, \mathbb{P})$ be a complete probability space endowed with a filtration $(\mathcal{F}_t)_{t \in [0,T]}$ satisfying $\mathcal F_t = \bigcap_{s > t} \mathcal F_s$ for all $t \in [0,T]$ and let $E$ be a separable Hilbert space.
We denote by $\mathcal L(E)$ the space of linear and bounded operators defined on $E$.
Let $Q \in \mathcal L(E)$ be a symmetric and nonnegative semidefinite operator such that $\text{Tr } Q < \infty$. 
Then we have the following definition.

\begin{definition}[Definition 4.2,\cite{seiid}]
An $E$-valued stochastic process $(W(t))_{t \in [0,T]}$ is called a \textbf{Q-Wiener process} if
\begin{itemize}
 \item $W(0) = 0$;
 \item $(W(t))_{t \in [0,T]}$ has continuous trajectories;
 \item $(W(t))_{t \in [0,T]}$ has independent increments;
 \item the distribution of $W(t) - W(s)$ is a Gaussian measure with mean 0 and covariance $(t-s)Q$ for $0 \leq s \leq t \leq T$.
\end{itemize}
\end{definition}

Next, we give a definition of $\mathcal F_t$-adapted processes and predictable processes, which are important to construct the stochastic integral.
Let $\mathcal P$ denote the smallest $\sigma$-field of subsets of $[0,T] \times \Omega$.

\begin{definition}[\cite{seiid}]
 A stochastic process $(X(t))_{t \in [0,T]}$ taking values in the measurable space $(\mathcal X, \mathscr B(\mathcal X))$ is called $\mathcal F_t$-adapted if for arbitrary $t \in [0,T]$ the random variable $X(t)$ is $\mathcal F_t$-measurable.
 We call $(X(t))_{t \in [0,T]}$ predictable if it is a measurable mapping from $([0,T] \times \Omega,\mathcal P)$ to $(\mathcal X, \mathscr B(\mathcal X))$.
\end{definition}

Every predictable stochastic process is $\mathcal F_t$-adapted.
The converse is in general not true. 
However, the following result is useful to conclude that a stochastic process has a predictable version.

\begin{lemma}[Proposition 3.7,\cite{seiid}] \label{predictable}
 Assume that the stochastic process $(X(t))_{t \in [0,T]}$ is $\mathcal F_t$-adapted and stochastically continuous.
 Then the process $(X(t))_{t \in [0,T]}$ has a predictable version.
\end{lemma}

Let $Q \in \mathcal L(E)$ be the covariance operator of a Q-Wiener process $(W(t))_{t \in [0,T]}$ with values in $E$.
Then there exists a unique operator $Q^{1/2} \in \mathcal L(E)$ such that $Q^{1/2} \circ Q^{1/2} = Q$.
We denote by $\mathcal{L}_{(HS)}(Q^{1/2}(E);\mathcal H)$ the space of Hilbert-Schmidt operators mapping from $Q^{1/2}(E)$ into another Hilbert space $\mathcal H$.
Let $(\Phi(t))_{t\in [0,T]}$ be a predictable process with values in $\mathcal{L}_{(HS)}(Q^{1/2}(E);\mathcal H)$ such that $ \mathbb E \int_0^T \left\| \Phi(t) \right\|_{\mathcal{L}_{(HS)}(Q^{1/2}(E);\mathcal H)}^2 dt < \infty$.
Then one can define the stochastic integral 
\begin{equation*}
 \psi(t) = \int\limits_0^t \Phi(s) d W(s)
\end{equation*}
for all $t \in [0,T]$ and we have
\begin{equation}\label{isometry}
 \mathbb E \left\| \psi(t)\right\|_{\mathcal H}^2 = \mathbb E \int\limits_0^t \left\| \Phi(s) \right\|_{\mathcal{L}_{(HS)}(Q^{1/2}(E);\mathcal H)}^2 ds.
\end{equation}

The following proposition is useful when dealing with a closed operator $\mathcal A \colon D(\mathcal A) \subset \mathcal H \rightarrow \mathcal H$.

\begin{proposition}[cf. Proposition 4.15,{\cite{seiid}}]\label{closedopstochint}
 If $\Phi(t)y \in D(\mathcal A)$ for every $y \in E$, all $t \in [0,T]$ and $\mathbb P$-almost surely,
 \begin{equation*}
  \mathbb E \int\limits_0^T \left\| \Phi(t) \right\|_{\mathcal{L}_{(HS)}(Q^{1/2}(E);\mathcal H)}^2 dt < \infty \quad \text{and} \quad 
  \mathbb E \int\limits_0^T \left\| \mathcal A \Phi(t) \right\|_{\mathcal{L}_{(HS)}(Q^{1/2}(E);\mathcal H)}^2 dt < \infty,
 \end{equation*}
 then we have $\mathbb P$-a.s. $\int_0^T \Phi(t) d W(t) \in D(\mathcal A)$ and
 \begin{equation*}
  \mathcal A\int\limits_0^T \Phi(t) d W(t) = \int\limits_0^T \mathcal A\Phi(t) d W(t).
 \end{equation*}
\end{proposition}

Next, we state a martingale representation theorem for Q-Wiener processes, which we use to construct solutions of backward SPDE's.
Let $Q \in \mathcal L(E)$ be the covariance operator of a Q-Wiener process $(W(t))_{t \in [0,T]}$.
Recall that the operator $Q \in \mathcal L(E)$ is a symmetric and nonnegative semidefinite such that $\text{Tr } Q < \infty$.
Hence, there exists a complete orthonormal system $(e_k)_{k \in \mathbb N}$ in $E$ and a bounded sequence of nonnegative real numbers $(\mu_k)_{k \in \mathbb N}$ such that $Q e_k = \mu_k e_k$ for each $k \in \mathbb N$.
Then for arbitrary $t \in [0,T]$, a Q-Wiener process has the expansion
\begin{equation*}
 W(t) = \sum\limits_{k=1}^\infty \sqrt{\mu_k} w_k(t) e_k,
\end{equation*}
where $(w_k(t))_{t \in [0,T]}$, $k \in \mathbb N$, are real valued mutually independent Brownian motions.
The convergence is in $L^2(\Omega)$.
Furthermore, we assume that the complete probability space $(\Omega,\mathcal F, \mathbb P)$ is endowed with the filtration $\mathcal F_t = \sigma \{ \bigcup_{k=1}^\infty \mathcal F_t^k\}$, where $\mathcal F_t^k = \sigma \{w_k(s): 0 \leq s \leq t\}$ for $t \in [0,T]$ and we require that the $\sigma$-algebra $\mathcal F$ satisfies $\mathcal F = \mathcal F_T$.
Then we have the following martingale representation theorem.

\begin{proposition}[Theorem 2.5,{\cite{sdei}}]\label{martingalerepresentation}
 Let the process $(M(t))_{t \in [0,T]}$ be a continuous $\mathcal F_t$-martingale with values in $\mathcal H$ such that $\mathbb E \| M(t)\|_{\mathcal H}^2 < \infty$ for all $t \in [0,T]$.
 Then there exists a unique predictable process $(\Phi(t))_{t \in [0,T]}$ with values in $\mathcal{L}_{(HS)}(Q^{1/2}(E);\mathcal H)$ such that $\mathbb E \int_0^T \| \Phi(t)\|_{\mathcal{L}_{(HS)}(Q^{1/2}(E);\mathcal H)}^2 dt < \infty$ and we have for all $t \in [0,T]$ and $\mathbb P$-a.s.
 \begin{equation*}
  M(t) = \mathbb E M(0) + \int\limits_0^t \Phi(s) d W(s).
 \end{equation*}
\end{proposition}

Finally, we state a product formula for infinite dimensional stochastic processes, which we use to obtain a duality principle.
The formula is an immediate consequence of the Itô formula, see \cite[Theorem 4.32]{seiid}.

\begin{lemma}\label{productformula}
 For $i=1,2$, assume that $X_i^0$ are $\mathcal F_0$-measurable $\mathcal H$-valued random variables, $(f_i(t))_{t \in [0,T]}$ are $\mathcal H$-valued predictable processes such that $\mathbb E \int_0^T \| f_i(t) \|_{\mathcal H} dt < \infty$, and $(\Phi_i(t))_{t \in [0,T]}$ are $\mathcal{L}_{(HS)}(Q^{1/2}(E);\mathcal H)$-valued predictable processes such that $\mathbb E \int_0^T \| \Phi_i(t) \|_{\mathcal{L}_{(HS)}(Q^{1/2}(E);\mathcal H)}^2 dt < \infty$. 
 For $i=1,2$, assume that the processes $(X_i(t))_{t \in [0,T]}$ satisfy for all $t \in [0,T]$ and $\mathbb P$-a.s.
 \begin{equation*}
  X_i(t) = X_i^0 + \int\limits_0^t f_i(s) ds + \int\limits_0^t \Phi_i(s) dW(s).
 \end{equation*}
 Then we have for all $t \in [0,T]$ and $\mathbb P$-a.s.
 \begin{align*}
  \left \langle X_1(t),X_2(t) \right\rangle_{\mathcal H} 
  &= \left \langle X_1^0,X_2^0 \right\rangle_{\mathcal H} + \int\limits_0^t \left[ \left \langle X_1(s),f_2(s) \right\rangle_{\mathcal H} + \left\langle X_2(s),f_1(s) \right\rangle_{\mathcal H} + \left\langle \Phi_1(s),\Phi_2(s) \right\rangle_{\mathcal{L}_{(HS)}(Q^{1/2}(E);\mathcal H)} \right] ds \\
  &\quad + \int\limits_0^t \left \langle X_1(s),\Phi_2(s) d W(s) \right\rangle_{\mathcal H} + \int\limits_0^t \left \langle X_2(s),\Phi_1(s) d W(s) \right\rangle_{\mathcal H}.
 \end{align*}
\end{lemma}

\section{The Stochastic Stokes Equations}\label{sec:stochasticstokes}

In this section, we consider the controlled stochastic Stokes equations.
Here, controls appear as distributed controls inside the domain as well as tangential controls on the boundary.

Let $(\Omega,\mathcal{F}, \mathbb{P})$ be a complete probability space endowed with a filtration $(\mathcal{F}_t)_{t \in [0,T]}$ satisfying $\mathcal F_t = \bigcap_{s > t} \mathcal F_s$ for all $t \in [0,T]$.
We assume that the external force $f(t)$ in equation (\ref{detstokes}) can be decomposed as the sum of a control term and a noise term dependent on the velocity field $y(t)$.
Using the spaces and operators introduced in Section \ref{sec:functionalbackground} and Section \ref{sec:stokesequations}, we obtain the stochastic Stokes equations:
\begin{equation}\label{stochstokes}
 \left\{
 \begin{aligned}
  d y(t) &= \left[- A y(t) + B u(t) + A D v(t) \right] dt + G(y(t)) d W(t), \\
  y(0) &= \xi,
 \end{aligned}
 \right.
\end{equation}
where the initial value $\xi$ is assumed to be $\mathcal F_0$-measurable and the process $(W(t))_{t \in [0,T]}$ is a Q-Wiener process with values in $H$ and covariance operator $Q \in \mathcal L(H)$.
The set of admissible distributed controls $U$ contains all predictable processes $(u(t))_{t \in [0,T]}$ with values in $H$ such that 
\begin{equation*}
 \mathbb E \int\limits_0^T \left\| u(t) \right\|_H^2 dt < \infty.
\end{equation*}
The space $U$ equipped with the inner product of $L^2(\Omega;L^2([0,T];H))$ becomes a Hilbert space.
Similarly, the set of admissible boundary controls $V$ contains all predictable processes $(v(t))_{t \in [0,T]}$ with values in $V^0(\partial \mathcal D)$ such that 
\begin{equation*}
 \mathbb E \int\limits_0^T \left\| v(t) \right\|_{V^0(\partial \mathcal D)}^2 dt < \infty.
\end{equation*}
The space $V$ equipped with the inner product of $L^2(\Omega;L^2([0,T];V^0(\partial \mathcal D)))$ becomes a Hilbert space.
The operators $B \colon H \rightarrow H$ and $G \colon H \rightarrow \mathcal{L}_{(HS)}(Q^{1/2}(H);H)$ are linear and bounded.
Motivated by Section \ref{sec:stokesequations}, we introduce the definition of a mild solution to system (\ref{stochstokes}).

\begin{definition}
 A predictable process $(y(t))_{t \in [0,T]}$ with values in $D(A^\alpha)$ is called a \textbf{mild solution of system (\ref{stochstokes})} if
 \begin{equation}\label{solutionregularity}
  \mathbb E \int\limits_0^T \| y(t)\|_{D(A^\alpha)}^2 dt < \infty,
 \end{equation}
 and we have for $t \in [0,T]$ and $\mathbb P$-a.s.
 \begin{align*}
  y(t) =&\; e^{-A t} \xi + \int\limits_0^t e^{-A (t-s)} B u(s) ds + \int\limits_0^t A e^{-A (t-s)} D v(s) ds + \int\limits_0^t e^{-A (t-s)}G(y(s)) d W(s).
 \end{align*}
\end{definition}

We get the following existence and uniqueness result.

\begin{theorem}\label{existencestochstokes}
 Let the controls $u \in U$ and $v \in V$ be fixed.
 If $\alpha \in [0,\frac{1}{4})$, then for any $\xi \in L^2(\Omega;D(A^\alpha))$, there exists a unique mild solution $(y(t))_{t \in [0,T]}$ of system (\ref{stochstokes}).
\end{theorem}

\begin{proof}
 For all $t_0,t_1 \in [0,T]$ with $t_0 < t_1$, let the space $\mathcal Z_{[t_0,t_1]}$ contain all predictable processes $(\tilde y(t))_{t \in [t_0,t_1]}$ with values in $D(A^\alpha)$ such that $\mathbb E \int_{t_0}^{t_1} \| \tilde y(t)\|_{D(A^\alpha)}^2 dt < \infty$.
 The space $\mathcal Z_{[t_0,t_1]}$ equipped with the inner product 
 \begin{equation*}
  \langle \tilde y_1,\tilde y_2 \rangle_{\mathcal Z_{[t_0,t_1]}}^2 = \mathbb E \int\limits_{t_0}^{t_1} \langle \tilde y_1(t),\tilde y_2(t) \rangle_{D(A^\alpha)}^2 dt
 \end{equation*}
 for every $\tilde y_1,\tilde y_2 \in \mathcal Z_{[t_0,t_1]}$ becomes a Hilbert space.
 We define for $t \in [0,T]$ and $\mathbb P$-a.s.
 \begin{equation*}
  \mathcal J(\tilde y)(t) = e^{-A t} \xi + \int\limits_0^t e^{-A (t-s)} B u(s) ds + \int\limits_0^t A e^{-A (t-s)} D v(s) ds + \int\limits_0^t e^{-A (t-s)}G(\tilde y(s)) d W(s).
 \end{equation*}
 Let $T_1 \in (0,T]$ and let us denote by $\mathcal Z_{T_1}$ the space $\mathcal Z_{[0,T_1]}$.
 First, we prove that $\mathcal J$ maps $\mathcal Z_{T_1}$ into itself.
 We define for $t \in [0,T_1]$ and $\mathbb P$-a.s.
 \begin{equation*}
  \psi_1(t) = e^{-A t} \xi + \int\limits_0^t e^{-A (t-s)} B u(s) ds,
  \quad \psi_2(t) = \int\limits_0^t A e^{-A (t-s)} D v(s) ds, 
  \quad \psi_3(\tilde y)(t) = \int\limits_0^t e^{-A (t-s)}G(\tilde y(s)) d W(s).
 \end{equation*}
 Recall that $\left\| e^{-At} \right\|_{\mathcal L(H)} \leq 1$ for all $t \in [0,T]$ and $B \colon H \rightarrow H$ is bounded.
 Using Lemma \ref{fractional} and the Cauchy-Schwarz inequality, the process $(\psi_1(t))_{t \in [0,T_1]}$ takes values in $D(A^\alpha)$ and there exists a constant $C_1 > 0$ such that
 \begin{align*}
  \mathbb E \int\limits_0^{T_1} \left\| \psi_1(t) \right\|_{D(A^\alpha)}^2 dt
  &\leq 2 \, \mathbb E \int\limits_0^{T_1} \left\| e^{-A t} A^\alpha \xi \right\|_H^2 dt + 2 \, \mathbb E \int\limits_0^{T_1} \left(\int\limits_0^t \left\| A^\alpha e^{-A (t-s)} B u(s) \right\|_H ds \right)^2 dt \\
  &\leq 2 T_1 \, \mathbb E \left\| \xi \right\|_{D(A^\alpha)}^2 + 2 M_\alpha^2 \, \mathbb E \int\limits_0^{T_1} \left(\int\limits_0^t (t-s)^{-\alpha} \left\| B u(s) \right\|_H ds \right)^2 dt \\
  &\leq C_1 \left[ \mathbb E \|\xi\|_{D(A^\alpha)}^2 + \mathbb E \int\limits_0^{T_1} \left\| u(t)\right\|_H^2 dt \right].
 \end{align*}
 Recall that $A^\beta D \colon V^0(\partial \mathcal D) \rightarrow H$ is bounded for all $\beta \in \left(0,\frac{1}{4}\right)$.
 We chose $\beta$ such that $\alpha < \beta$.
 By Lemma \ref{fractional} and Young's inequality for convolutions, the process $(\psi_2(t))_{t \in [0,T_1]}$ takes values in $D(A^\alpha)$ and there exists a constant $C_2 > 0$ such that
 \begin{align*}
  \mathbb E \int\limits_0^{T_1} \left\| \psi_2(t) \right\|_{D(A^\alpha)}^2 dt
  &\leq \mathbb E \int\limits_0^{T_1} \left(\int\limits_0^t \left\| A^{1+\alpha-\beta} e^{-A (t-s)} A^\beta D v(s) \right\|_H ds \right)^2 dt \\
  &\leq M_{1+\alpha-\beta}^2 \mathbb E \int\limits_0^{T_1} \left(\int\limits_0^t (t-s)^{-1-\alpha+\beta} \left\| A^\beta D v(s) \right\|_H ds \right)^2 dt \\
  &\leq M_{1+\alpha-\beta}^2 \left(\int\limits_0^{T_1} t^{-1-\alpha+\beta} dt \right)^2 \mathbb E \int\limits_0^{T_1} \left\| A^\beta D v(t) \right\|_H^2 dt \\
  &\leq C_2 \, \mathbb E \int\limits_0^{T_1} \left\| v(t)\right\|_{V^0(\partial \mathcal D)}^2 dt.
 \end{align*}
 Due to Lemma \ref{fractional} and the fact that the operator $G \colon H \rightarrow \mathcal{L}_{(HS)}(Q^{1/2}(H);H)$ is bounded, one can verify the assumptions of Proposition \ref{closedopstochint} with $\mathcal A = A^\alpha$ and hence, the process $(\psi_3(\tilde y)(t))_{t \in [0,T_1]}$ takes values in $D(A^\alpha)$.
 Using Lemma \ref{fractional}, Fubini's theorem, the Itô isometry (\ref{isometry}) and Young's inequality for convolutions, there exists a constant $C_3 > 0$ such that
 \begin{align}\label{ineq1}
  \mathbb E \int\limits_0^{T_1} \left\| \psi_3(\tilde y)(t) \right\|_{D(A^\alpha)}^2 dt
  &= \int\limits_0^{T_1} \mathbb E \int\limits_0^t \left\| A^\alpha e^{-A (t-s)} G(\tilde y(s)) \right\|_{\mathcal{L}_{(HS)}(Q^{1/2}(H);H)}^2 ds \, dt \nonumber \\
  &\leq M_\alpha^2 \, \mathbb E \int\limits_0^{T_1} \int\limits_0^t (t-s)^{-2\alpha} \left\| G(\tilde y(s)) \right\|_{\mathcal{L}_{(HS)}(Q^{1/2}(H);H)}^2 ds \, dt \nonumber \\
  &\leq C_3 T_1^{1-2\alpha} \, \mathbb E \int\limits_0^{T_1} \| \tilde y(t)\|_{D(A^\alpha)}^2 dt.
 \end{align}
 Hence, we can conclude that for fixed $\tilde y \in \mathcal Z_{T_1}$, the process $(\mathcal J(\tilde y)(t))_{t \in [0,T_1]}$ takes values in $D(A^\alpha)$ such that $\mathbb E \int_0^{T_1} \| \mathcal J(\tilde y)(t)\|_{D(A^\alpha)}^2 < \infty$.
 Obviously, the process $(\mathcal J(\tilde y)(t))_{t \in [0,T_1]}$ is predictable.
 We conclude that $\mathcal J$ maps $\mathcal Z_{T_1}$ into itself.
 
 Next, we show that $\mathcal J$ is a contraction on $\mathcal Z_{T_1}$.
 Recall that the operator $G \colon H \rightarrow \mathcal{L}_{(HS)}(Q^{1/2}(H);H)$ is linear.
 Using inequality (\ref{ineq1}), we get for every $ \tilde y_1, \tilde y_2 \in \mathcal Z_{T_1}$
 \begin{equation*}
  \mathbb E \int\limits_0^{T_1} \left\| \mathcal J(\tilde y_1)(t) - \mathcal J(\tilde y_2)(t) \right\|_{D(A^\alpha)}^2 dt
  = \mathbb E \int\limits_0^{T_1} \left\| \psi_3(\tilde y_1-\tilde y_2)(t) \right\|_{D(A^\alpha)}^2 dt
  \leq C_3 T_1^{1-2\alpha} \; \mathbb E \int\limits_0^{T_1} \| \tilde y_1(t)-\tilde y_2(t)\|_{D(A^\alpha)}^2 dt.
 \end{equation*}
 We choose $T_1 \in (0,T]$ such that $C_3 T_1^{1-2\alpha}<1$.
 Applying Banach fixed point theorem, we get a unique element $y \in \mathcal Z_{T_1}$ such that for $t \in [0,T_1]$ and $\mathbb P$-a.s. $y(t) = \mathcal J(y)(t)$.
 
 Next, we consider for $t \in [T_1,T]$ and $\mathbb P$-a.s.
 \begin{equation*}
  \mathcal J(\tilde y)(t) = e^{-A (t-T_1)} y(T_1) + \int\limits_{T_1}^t e^{-A (t-s)} B u(s) ds + \int\limits_{T_1}^t A e^{-A (t-s)} D v(s) ds + \int\limits_{T_1}^t e^{-A (t-s)}G(\tilde y(s)) d W(s).
 \end{equation*}
 Again, for a certain $T_2 \in [T_1,T]$, there exists a unique fixed point of $\mathcal J$ on $\mathcal Z_{[T_1,T_2]}$.
 By continuing the method, we get the existence and uniqueness of a predictable process $(y(t))_{t \in [0,T]}$ satisfying for $t \in [0,T]$ and $\mathbb P$-a.s. $y(t) = \mathcal J(y)(t)$.
\end{proof}

For the rest of the paper, we assume that $(y(t))_{t \in [0,T]}$ satisfies condition (\ref{solutionregularity}) with $\alpha=0$ and we assume that the initial value $\xi \in L^2(\Omega;H)$ is fixed.
To illustrate the dependence on the controls $u \in U$ and $v \in V$, let us denote by $(y(t;u,v))_{t \in [0,T]}$ the mild solution of system (\ref{stochstokes}).
Whenever the process is considered for fixed controls, we omit the dependency.

Next, we show some useful properties.
Therefor, we need the following formulation of Gronwall's inequality for integrable functions.
The result might be deduced from more general formulations, see \cite{pfqa,anto,aggi}.

\begin{lemma}\label{gronwall}
 Let $a,x \colon [0,T] \rightarrow [0,\infty)$ be integrable functions and let $b \geq 0$.
 If
 \begin{equation*}
  x(t) \leq a(t) + b \int\limits_0^t x(s) ds
 \end{equation*}
 for all $t \in [0,T]$, then
 \begin{equation*}
  x(t) \leq a(t) + b \int\limits_0^t e^{b(t-s)} a(s) ds.
 \end{equation*}
 for all $t \in [0,T]$.
 If $a(t)$ is nondecreasing on $[0,T]$, then for $t \in [0,T]$
 \begin{equation*}
  x(t) \leq a(t) e^{b t}.
 \end{equation*}
\end{lemma}

\begin{corollary}\label{stateproperties}
 Let $(y(t;u,v))_{t \in [0,T]}$ be the mild solution of system (\ref{stochstokes}) corresponding to the controls $u \in U$ and $v \in V$.
 Then the process $(y(t;u,v))_{t \in [0,T]}$ is affine linear with respect to $u$ and $v$, and we have for every $u_1,u_2 \in U$ and every $v_1,v_2 \in V$
 \begin{equation}\label{statecontinuity}
  \mathbb E \int\limits_0^T \| y(t;u_1,v_1) - y(t;u_2,v_2)\|_H^2 dt \leq \widehat C \left[ \mathbb E \int\limits_0^T \left\| u_1(t) - u_2(t) \right\|_H^2 dt + \mathbb E \int\limits_0^T \left\| v_1(t) - v_2(t) \right\|_{V^0(\partial \mathcal D)}^2 dt \right],
 \end{equation}
 where $\widehat C > 0$ is a constant.
\end{corollary}

\begin{proof}
 First, we show that $(y(t;u,v))_{t \in [0,T]}$ is affine linear with respect to $u \in U$.
 We assume that $\xi=0$ and $v=0$.
 Moreover, let $a,b \in \mathbb R$ and $u_1,u_2 \in U$.
 Recall that the operators $B \colon H \rightarrow H$ and $G \colon H \rightarrow \mathcal{L}_{(HS)}(Q^{1/2}(H);H)$ are linear and bounded.
 Moreover, we have $\left\| e^{-At} \right\|_{\mathcal L(H)} \leq 1$ for all $t \in [0,T]$. 
 Using the Itô isometry (\ref{isometry}) and Fubini's theorem, there exists a constant $C^*>0$ such that for $t \in [0,T]$
 \begin{align*}
  &\mathbb E \, \| y(t;a \, u_1 + b \, u_2,0) - a \, y(t;u_1,0) - b \, y(t;u_2,0) \|_H^2 \\
  &\leq \mathbb E \, \left\| \int\limits_0^t e^{-A (t-s)} G(y(t;a \, u_1 + b \, u_2,0) - a \, y(t;u_1,0) - b \, y(t;u_2,0)) d W(s) \right\|_H^2 \\
  &\leq C^* \int\limits_0^t \mathbb E \left\| y(s;a \, u_1 + b \, u_2,0) - a \, y(s;u_1,0) - b \, y(s;u_2,0)) \right\|_H^2 ds.
 \end{align*}
 By Lemma \ref{gronwall} and Fubini's theorem, we get
 \begin{equation*}
  \mathbb E \int\limits_0^T \| y(t;a \, u_1 + b \, u_2,0) - a \, y(t;u_1,0) - b \, y(t;u_2,0) \|_H^2 dt = 0.
 \end{equation*}
 We obtain that $(y(t;u,0))_{t \in [0,T]}$ with initial value $\xi = 0$ is linear with respect to $u \in U$.
 For arbitrary $\xi \in L^2(\Omega;H)$ and $v \in V$, we can conclude that $(y(t;u,v))_{t \in [0,T]}$ is affine linear with respect to $u \in U$.
 Similarly, we obtain that $(y(t;u,v))_{t \in [0,T]}$ is affine linear with respect to $v \in V$.
 
 Next, we show that inequality (\ref{statecontinuity}) holds.
 Let $u_1,u_2 \in U$ and $v_1,v_2 \in V$.
 Recall that $A^\alpha D \colon V^0(\partial \mathcal D) \rightarrow H$ is linear and bounded for all $\alpha \in \left(0,\frac{1}{4}\right)$.
 Due to the Itô isometry (\ref{isometry}), Lemma \ref{fractional} and Fubini's theorem, there exist constants $C_1,C_2,C_3>0$ such that for $t \in [0,T]$
 \begin{align*}
  \mathbb E \, \| y(t;u_1,v_1) - y(t;u_2,v_2) \|_H^2
  &\leq C_1 \, \mathbb E \int\limits_0^t \left\| u_1(s) - u_2(s) \right\|_H^2 ds + C_2 \, \mathbb E \left( \int\limits_0^t (t-s)^{\alpha-1} \left\| v_1(s) - v_2(s) \right\|_{V^0(\partial \mathcal D)} ds \right)^2 \\
  &\quad + C_3 \int\limits_0^t \mathbb E \, \| y(s;u_1,v_1) - y(s;u_2,v_2) \|_H^2 ds.
 \end{align*}
 Using Lemma \ref{gronwall}, Fubini's theorem and Young's inequality for convolutions, we get for $t \in [0,T]$
 \begin{align*}
  &\mathbb E \, \| y(t;u_1,v_1) - y(t;u_2,v_2) \|_H^2 \\
  &\leq C_1 \, \mathbb E \int\limits_0^t \left\| u_1(s) - u_2(s) \right\|_H^2 ds + C_2 \, \mathbb E \left( \int\limits_0^t (t-s)^{\alpha-1} \left\| v_1(s) - v_2(s) \right\|_{V^0(\partial \mathcal D)} ds \right)^2 \\
  &\quad + C_3 \int\limits_0^t e^{C_3(t-s)} \left[ C_1 \, \mathbb E \int\limits_0^s \left\| u_1(r) - u_2(r) \right\|_H^2 dr + C_2 \, \mathbb E \left( \int\limits_0^s (s-r)^{\alpha-1} \left\| v_1(r) - v_2(r) \right\|_{V^0(\partial \mathcal D)} dr \right)^2 \right] ds \\
  &\leq C_1 \left(1 + C_3 e^{C_3 t}\right) \, \mathbb E \int\limits_0^t \left\| u_1(s) - u_2(s) \right\|_H^2 ds + C_2 \, \mathbb E \left( \int\limits_0^t (t-s)^{\alpha-1} \left\| v_1(s) - v_2(s) \right\|_{V^0(\partial \mathcal D)} ds \right)^2 \\
  &\quad + \frac{C_2 C_3 e^{C_3 t} t^{2\alpha}}{\alpha^2} \; \mathbb E \int\limits_0^t \left\| v_1(s) - v_2(s) \right\|_{V^0(\partial \mathcal D)}^2 ds.
 \end{align*}
 By Fubini's theorem and Young's inequality for convolutions, there exists a constant $\widehat C>0$ such that
 \begin{align*}
  &\mathbb E \int\limits_0^T \| y(t;u_1,v_1) - y(t;u_2,v_2) \|_H^2 dt \\
  &\leq \int\limits_0^T \left[ C_1 \left(1 + C_3 e^{C_3 t}\right) \, \mathbb E \int\limits_0^t \left\| u_1(s) - u_2(s) \right\|_H^2 ds + C_3 e^{C_3 t} \frac{C_2 t^{2\alpha}}{\alpha^2} \; \mathbb E \int\limits_0^t \left\| v_1(s) - v_2(s) \right\|_{V^0(\partial \mathcal D)}^2 ds\right] dt \\
  &\quad + C_2 \, \mathbb E \int\limits_0^T \left( \int\limits_0^t (t-s)^{\alpha-1} \left\| v_1(s) - v_2(s) \right\|_{V^0(\partial \mathcal D)} ds \right)^2 dt \\
  &\leq \widehat C \left[ \mathbb E \int\limits_0^T \left\| u_1(t) - u_2(t) \right\|_H^2 dt + \mathbb E \int\limits_0^T \left\| v_1(t) - v_2(t) \right\|_{V^0(\partial \mathcal D)}^2 dt \right].
 \end{align*}
\end{proof}

\section{The Control Problem}

The control problem considered in this paper is motivated by \cite{apfo,mawr,epfa,fbso,ocot}.
In this section, we state first order optimality conditions, which are necessary and sufficient.
Moreover, we derive a duality principle such that we can deduce explicit formulas the optimal controls have to satisfy.

Let us introduce the following cost functional:
\begin{equation}\label{costfunctional}
 J(u,v) = \frac{1}{2} \, \mathbb E \int\limits_0^T \| y(t;u,v) - y_d(t) \|_H^2 dt + \frac{\kappa_1}{2} \, \mathbb E \int\limits_0^T \left\| u(t) \right\|_H^2 dt + \frac{\kappa_2}{2} \, \mathbb E \int\limits_0^T \left\| v(t) \right\|_{V^0(\partial \mathcal D)}^2 dt,
\end{equation}
where $(y(t;u,v))_{t \in [0,T]}$ is the mild solution of system (\ref{stochstokes}) corresponding to the controls $u \in U$ and $v \in V$.
The function $y_d \in L^2([0,T];H)$ is a given desired velocity field and $\kappa_1,\kappa_2 > 0$ are weights.
The task is to find controls $\overline u \in U$ and $\overline v \in V$ such that
\begin{equation*}
 J(\overline u,\overline v) = \inf_{u \in U, v \in V} J(u,v).
\end{equation*}
The controls $\overline u \in U$ and $\overline v \in V$ are called optimal controls.
Note that the control problem is formulated as an unbounded optimization problem constrained by a SPDE.
The functional $J \colon U \times V \rightarrow \mathbb R$ given by equation (\ref{costfunctional}) is continuous, coercive and strictly convex, which is a consequence of Corollary \ref{stateproperties}.
Hence, we get the existence and uniqueness of optimal controls.
For more details, we refer to \cite{cfa,nfaa}.

\subsection{Necessary and Sufficient Optimality Conditions}

First, let us introduce the following systems:
\begin{equation}\label{frechet1}
 \left\{
 \begin{aligned}
  d z_1(t) &= \left[- A z_1(t) + B u(t) \right] dt + G(z_1(t)) d W(t), \\
  z_1(0) &= 0,
 \end{aligned}
 \right.
\end{equation}
\begin{equation}\label{frechet2}
 \left\{
 \begin{aligned}
  d z_2(t) &= \left[- A z_2(t) + A D v(t) \right] dt + G(z_2(t)) d W(t), \\
  z_2(0) &= 0,
 \end{aligned}
 \right.
\end{equation}
where $u \in U$, $v \in V$ and $(W(t))_{t \in [0,T]}$ is a Q-Wiener process with values in $H$ and covariance operator $Q \in \mathcal L(H)$.
The operators $A,B,D,G$ and the spaces $U,V$ are introduced in Section \ref{sec:preliminaries} and Section \ref{sec:stochasticstokes}, respectively.

\begin{definition}
 a) A predictable process $(z_1(t))_{t \in [0,T]}$ with values in $H$ is called a \textbf{mild solution of system (\ref{frechet1})} if
 \begin{equation}\label{regularity}
  \mathbb E \int\limits_0^T \| z_1(t)\|_H^2 dt < \infty,
 \end{equation}
 and we have for $t \in [0,T]$ and $\mathbb P$-a.s.
 \begin{equation*}
  z_1(t) = \int\limits_0^t e^{-A (t-s)} B u(s) ds + \int\limits_0^t e^{-A (t-s)}G(z_1(s)) d W(s).
 \end{equation*}
 b) A predictable process $(z_2(t))_{t \in [0,T]}$ with values in $H$ is called a \textbf{mild solution of system (\ref{frechet2})} if
 \begin{equation*}
  \mathbb E \int\limits_0^T \| z_2(t)\|_H^2 dt < \infty,
 \end{equation*}
 and we have for $t \in [0,T]$ and $\mathbb P$-a.s.
 \begin{equation*}
  z_2(t) = \int\limits_0^t A e^{-A (t-s)} D v(s) ds + \int\limits_0^t e^{-A (t-s)}G(z_2(s)) d W(s).
 \end{equation*}
\end{definition}

Existence and uniqueness results of mild solutions to system (\ref{frechet1}) and system (\ref{frechet2}) can be obtained similarly to Theorem \ref{existencestochstokes}.
For stronger regularity properties of the mild solution to system (\ref{frechet1}), we refer to \cite{seiid,sdei}.
However, we assume that the weaker condition (\ref{regularity}) holds.
To illustrate the dependence on the controls $u \in U$ and $v \in V$, let us denote by $(z_1(t;u))_{t \in [0,T]}$ and $(z_2(t;v))_{t \in [0,T]}$ the mild solutions of system (\ref{frechet1}) and system (\ref{frechet2}), respectively.
Whenever these processes are considered for fixed controls, we omit the dependency.
Similarly to Corollary \ref{stateproperties}, we get the following result.

\begin{lemma}\label{frechetproperties}
 Let $(z_1(t;u))_{t \in [0,T]}$ and $(z_2(t;v))_{t \in [0,T]}$ be the mild solutions of system (\ref{frechet1}) and system (\ref{frechet2}) corresponding to the controls $u \in U$ and $v \in V$, respectively.
 Then the process $(z_1(t;u))_{t \in [0,T]}$ is linear with respect to $u$ and the process $(z_2(t;v))_{t \in [0,T]}$ is linear with respect to $v$.
 Moreover, we have for every $u_1,u_2 \in U$ and every $v_1,v_2 \in V$
 \begin{align*}
  &\mathbb E \int\limits_0^T \| z_1(t;u_1) - z_1(t;u_2)\|_H^2 \leq \widehat C \, \mathbb E \int\limits_0^T \left\| u_1(t) - u_2(t) \right\|_H^2 dt, \\
  &\mathbb E \int\limits_0^T \| z_2(t;v_1) - z_2(t;v_2)\|_H^2 \leq \widehat C \, \mathbb E \int\limits_0^T \left\| v_1(t) - v_2(t) \right\|_{V^0(\partial \mathcal D)}^2 dt,
 \end{align*}
 where $\widehat C > 0$ is a constant.
\end{lemma}

Next, we calculate the Fr\'echet derivative of the mild solution to system (\ref{stochstokes}).
Let $X,Y$ and $Z$ be arbitrary Banach spaces.
For a mapping $f \colon M_X \times M_Y \rightarrow Z$ with $M_X \subset X$, $M_Y \subset Y$ nonempty and open, the (partial) Fr\'echet derivative at $x \in M_X$ in direction $h \in X$ for fixed $y \in Y$ is denoted by $d_x f(x,y) [h]$.
Analogously, the (partial) Fr\'echet derivative at $y \in M_Y$ in direction $h \in Y$ for fixed $x \in X$ is denoted by $d_y f(x,y) [h]$.
We get the following result.

\begin{theorem}
 Let $(y(t;u,v))_{t \in [0,T]}$, $(z_1(t;u))_{t \in [0,T]}$ and $(z_2(t;v))_{t \in [0,T]}$ be the mild solutions of systems (\ref{stochstokes}), (\ref{frechet1}) and (\ref{frechet2}) corresponding to the controls $u \in U$ and $v \in V$, respectively.
 Then the Fr\'echet derivative of $y(t;u,v)$ at $u \in U$ in direction $\tilde u \in U$ satisfies for fixed $v \in V$, $t \in [0,T]$ and $\mathbb P$-a.s.
 \begin{equation*}
  d_u y(t;u,v) [\tilde u] = z_1(t;\tilde u).
 \end{equation*}
 The Fr\'echet derivative of $y(t;u,v)$ at $v \in V$ in direction $\tilde v \in V$ satisfies for fixed $u \in U$, $t \in [0,T]$ and $\mathbb P$-a.s.
 \begin{equation*}
  d_v y(t;u,v) [\tilde v] = z_2(t;\tilde v).
 \end{equation*}
\end{theorem}

\begin{proof}
 First, we calculate the Fr\'echet derivative of $y(t;u,v)$ at $u \in U$ in direction $\tilde u \in U$.
 Let $v \in V$ be fixed.
 Recall that the operators $B \colon H \rightarrow H$ and $G \colon H \rightarrow \mathcal{L}_{(HS)}(Q^{1/2}(H);H)$ are linear and bounded.
 Moreover, we have $\left\| e^{-At} \right\|_{\mathcal L(H)} \leq 1$ for all $t \in [0,T]$.
 Using the Itô isometry (\ref{isometry}) and Fubini's theorem, there exists a constant $C^* > 0$ such that for $t \in [0,T]$
 \begin{align*}
  \mathbb E \, \| y(t;u + \tilde u,v) - y(t;u,v) - z_1(t;\tilde u)\|_H^2 
  &= \mathbb E \left\| \int\limits_0^t e^{-A (t-s)} G(y(s;u + \tilde u,v) - y(s;u,v) - z_1(s;\tilde u)) d W(s) \right\|_H^2 \\
  &\leq C^* \int\limits_0^t \mathbb E \, \left\| y(s;u + \tilde u,v) - y(s;u,v) - z_1(s;\tilde u) \right\|_H^2 ds.
 \end{align*}
 By Lemma \ref{gronwall} and Fubini's theorem, we get
 \begin{equation*}
  \mathbb E \int\limits_0^T \| y(t;u + \tilde u,v) - y(t;u,v) - z_1(t;\tilde u)\|_H^2 dt = 0.
 \end{equation*}
 Hence, the Fr\'echet derivative of $y(t;u,v)$ at $u \in U$ in direction $\tilde u \in U$ satisfies for every $v \in V$, $t \in [0,T]$ and $\mathbb P$-a.s.
 \begin{equation*}
  d_u y(t;u,v) [\tilde u] = z_1(t;\tilde u).
 \end{equation*}
 Due to Lemma \ref{frechetproperties}, the operator $d_u y(t;u,v)$ is linear and bounded on $U$.
 Similarly, we obtain the Fr\'echet derivative of $y(t;u,v)$ at $v \in V$ in direction $\tilde v \in V$.
\end{proof}

As a direct consequence of the previous theorem and the chain rule for Fr\'echet derivatives, we get the following result.

\begin{theorem}\label{costfunctionalfrechet}
 Let the functional $J \colon U \times V \rightarrow \mathbb R$ be defined by (\ref{costfunctional}).
 Then the Fr\'echet derivative at $u \in U$ in direction $\tilde u \in U$ for fixed $v \in V$ satisfies 
 \begin{equation*}
  d_u J(u,v)[\tilde u] = \mathbb E \int\limits_0^T \left\langle y(t;u,v) -y_d(t), z_1(t;\tilde u) \right\rangle_H dt + \kappa_1\, \mathbb E \int\limits_0^T \left\langle u(t), \tilde u(t) \right\rangle_H dt,
 \end{equation*}
 where $(z_1(t;\tilde u))_{t \in [0,T]}$ is the mild solution of system (\ref{frechet1}) corresponding to the control $\tilde u \in U$.
 The Fr\'echet derivative at $v \in V$ in direction $\tilde v \in V$ for fixed $u \in U$ satisfies 
 \begin{equation*}
  d_v J(u,v)[\tilde v] = \mathbb E \int\limits_0^T \left\langle y(t;u,v) -y_d(t), z_2(t;\tilde v) \right\rangle_H dt + \kappa_2\, \mathbb E \int\limits_0^T \left\langle v(t), \tilde v(t) \right\rangle_{V^0(\partial \mathcal D)} dt,
 \end{equation*}
 where $(z_2(t;\tilde v))_{t \in [0,T]}$ is the mild solution of system (\ref{frechet2}) corresponding to the control $\tilde v \in V$.
\end{theorem}

As a result of the previous theorem and the fact that the cost functional $J \colon U \times V \rightarrow \mathbb R$ given by (\ref{costfunctional}) is strictly convex, the optimal controls $\overline u \in U$ and $\overline v \in V$ satisfy the following necessary and sufficient optimality conditions:
\begin{align}
 d_u J(\overline u,\overline v)[\tilde u] &= 0, \label{optcondition1}\\
 d_v J(\overline u,\overline v)[\tilde v] &= 0 \label{optcondition2}
\end{align}
for every $\tilde u \in U$ and every $\tilde v \in V$.
For more details about optimality conditions of convex differentiable functionals, we refer to \cite{cfa,nfaa}.
Next, we use the optimality conditions (\ref{optcondition1}) and (\ref{optcondition2}) to derive explicit formulas for the optimal controls $\overline u \in U$ and $\overline v \in V$.
Therefor, we need a duality principle, which gives us a relation between the Fr\'echet derivatives of the mild solution to system (\ref{stochstokes}) and the adjoint equation, which is given by a backward SPDE.

\subsection{The Adjoint Equation}\label{sec:adjoint}

We introduce the following backward SPDE:
\begin{equation}\label{backwardstochstokes}
 \left\{
 \begin{aligned}
  d z^*(t) &= -[ -A z^*(t) + G^*(\Phi(t)) + y(t) - y_d(t)] dt + \Phi(t) d W(t), \\
  z^*(T) &= 0,
 \end{aligned}
 \right.
\end{equation}
where $(y(t))_{t \in [0,T]}$ is the mild solution of system (\ref{stochstokes}) and $y_d \in L^2([0,T];H)$ is the desired velocity field.
The process $(W(t))_{t \in [0,T]}$ is a Q-Wiener process with values in $H$ and covariance operator $Q \in \mathcal L(H)$ and the operator $G^* \colon \mathcal{L}_{(HS)}(Q^{1/2}(H);H) \rightarrow H$ is linear and bounded.
A precise meaning is given in the following remark.

\begin{remark}\label{backwardremarks}
 Since the operator $G \colon H \rightarrow \mathcal{L}_{(HS)}(Q^{1/2}(H);H)$ is linear and bounded, there exists a linear and bounded operator $G^* \colon \mathcal{L}_{(HS)}(Q^{1/2}(H);H) \rightarrow H$ satisfying for every $h \in H$ and every $\Phi \in \mathcal{L}_{(HS)}(Q^{1/2}(H);H)$
 \begin{equation}\label{adjoint}
  \langle G(h), \Phi \rangle_{\mathcal{L}_{(HS)}(Q^{1/2}(H);H)} = \langle h, G^*(\Phi) \rangle_H.
 \end{equation}
\end{remark}

\begin{definition}
 A pair of predictable processes $(z^*(t),\Phi(t))_{t \in [0,T]}$ with values in $H \times \mathcal{L}_{(HS)}(Q^{1/2}(H);H)$ is called a \textbf{mild solution of system (\ref{backwardstochstokes})} if
 \begin{align*}
  &\sup_{t \in [0,T]} \mathbb E \, \|z^*(t)\|_H^2 < \infty,
  &\mathbb E \int\limits_0^T \| \Phi(t)\|_{\mathcal{L}_{(HS)}(Q^{1/2}(H);H)}^2 dt < \infty,
 \end{align*}
 and we have for all $t \in [0,T]$ and $\mathbb P$-a.s.
 \begin{align*}
  z^*(t) = \int\limits_t^T e^{-A (s-t)} G^*(\Phi(s)) ds + \int\limits_t^T e^{-A (s-t)} \left( y(s) - y_d(s)\right) ds - \int\limits_t^T e^{-A (s-t)} \Phi(s) d W(s).
 \end{align*}
\end{definition}

An existence and uniqueness result is mainly based on the following lemma.

\begin{lemma}[Lemma 2.1,{\cite{asoa}}]\label{backwardlemma}
 Let $z \in L^2(\Omega;H)$ be $\mathcal F_T$-measurable and let $(f(t))_{t \in [0,T]}$ be a predictable process with values in $H$ such that $\mathbb E \int_0^T \|f(t)\|_H^2 dt < \infty$.
 Then there exists a unique pair of predictable processes $(\varphi(t),\phi(t))_{t \in [0,T]}$ with values in $H \times \mathcal{L}_{(HS)}(Q^{1/2}(H);H)$ such that for all $t \in [0,T]$ and $\mathbb P$-a.s.
 \begin{equation*}
  \varphi(t) = e^{-A (T-t)} z + \int\limits_t^T e^{-A (s-t)} f(s) ds - \int\limits_t^T e^{-A (s-t)} \phi(s) d W(s).
 \end{equation*}
 Moreover, there exists a constant $c>0$ such that for all $t \in [0,T]$
 \begin{align}
  &\mathbb E \, \|\varphi(t)\|_H^2 \leq c \left[ \mathbb E \, \| z \|_H^2 + (T-t) \, \mathbb E \int\limits_t^T \| f(s) \|_H^2 ds\right], \label{backwardinequality1}\\
  &\mathbb E \int\limits_t^T \| \phi(s) \|_{\mathcal{L}_{(HS)}(Q^{1/2}(H);H)}^2 ds \leq c \left[ \mathbb E \, \| z \|_H^2 + (T-t) \, \mathbb E \int\limits_t^T \| f(s) \|_H^2 ds\right]. \label{backwardinequality2}
 \end{align}
\end{lemma}

Existence and uniqueness results of mild solutions to backward SPDE's with cylindrical Wiener processes can be found in \cite{asoa}.
Similarly, we get the existence of a unique mild solution to system (\ref{backwardstochstokes}).
Furthermore, note that the mild solution of system (\ref{stochstokes}) depends on the controls $u \in U$ and $v \in V$.
Thus, we get this property for the mild solution of system (\ref{backwardstochstokes}) as well.
To illustrate the dependence on the controls $u \in U$ and $v \in V$, let us denote by $(z^*(t;u,v),\Phi(t;u,v))_{t \in [0,T]}$ the mild solution of system (\ref{backwardstochstokes}).
Whenever these processes are considered for fixed controls, we omit the dependency.
For the process $(z^*(t;u,v))_{t \in [0,T]}$, one can show another important regularity property.
Therefor, we need a modification of Young's inequality for convolutions.

\begin{lemma}\label{backwardyoung}
 Let $f \in L^p([0,T])$ and $g \in L^q([0,T])$ be arbitrary.
 We set for $t \in [0,T]$
 \begin{equation*}
  h(t) = \int\limits_t^T f(s-t) g(s) ds.
 \end{equation*}
 If $p,q,r \geq 1$ satisfy $\frac{1}{p} + \frac{1}{q} = \frac{1}{r} + 1$, then $h \in L^r([0,T])$ and
 \begin{equation*}
  \|h\|_{L^r([0,T])} \leq \|f\|_{L^p([0,T])} \|g\|_{L^q([0,T])}.
 \end{equation*}
\end{lemma}

\begin{proof}
 The proof can be obtained similarly to the classical version of Young's inequality for convolutions, see \cite[Theorem 3.9.4]{mt}.
\end{proof}

\begin{proposition}\label{backwardregularity}
 Let $(z^*(t;u,v),\Phi(t;u,v))_{t \in [0,T]}$ be the mild solution of system (\ref{backwardstochstokes}) corresponding to the controls $u \in U$ and $v \in V$.
 Then $(z^*(t;u,v))_{t \in [0,T]}$ takes values in $D(A^\varepsilon)$ with $\varepsilon \in [0,1)$ such that 
 \begin{equation*}
  \mathbb E \int\limits_0^T \|z^*(t;u,v)\|_{D(A^\varepsilon)}^2 dt < \infty.
 \end{equation*}
\end{proposition}

\begin{proof}
 For the sake of simplicity, we omit the dependence on the controls.
 Since $(z^*(t;u,v))_{t \in [0,T]}$ is predictable, we get for $t \in [0,T]$ and $\mathbb P$-a.s.
 \begin{equation*}
  z^*(t;u,v) = \mathbb E \left[ \left. \int\limits_t^T e^{-A (s-t)} G^*(\Phi(s)) ds + \int\limits_t^T e^{-A (s-t)} \left( y(s) - y_d(s)\right) ds \right| \mathcal F_t\right].
 \end{equation*}
 Recall that the operator $G^* \colon \mathcal{L}_{(HS)}(Q^{1/2}(H);H) \rightarrow H$ is bounded.
 Using Lemma \ref{fractional} and Lemma \ref{backwardyoung}, the process $(z^*(t))_{t \in [0,T]}$ takes values in $D(A^\varepsilon)$ with $\varepsilon \in [0,1)$ and there exists a constant $C^*>0$ such that
 \begin{align*}
  &\mathbb E \int\limits_0^T \|z^*(t;u,v)\|_{D(A^\varepsilon)}^2 dt \\
  &\leq 2 \, \mathbb E \int\limits_0^T \left( \int\limits_t^T \| A^\varepsilon e^{-A (s-t)} G^*(\Phi(s)) \|_H ds \right)^2 dt + 2 \, \mathbb E \int\limits_0^T \left( \int\limits_t^T \| A^\varepsilon e^{-A (s-t)} \left( y(s;u,v) - y_d(s)\right) \|_H^2 ds \right)^2 dt \\
  &\leq 2 M_\varepsilon^2 \, \mathbb E \int\limits_0^T \left(\int\limits_t^T (s-t)^{-\varepsilon} \| G^*(\Phi(s)) \|_H ds \right)^2 dt + 2 M_\varepsilon^2 \, \mathbb E \int\limits_0^T \left( \int\limits_t^T (s-t)^{-\varepsilon} \| y(s;u,v) - y_d(s) \|_H ds \right)^2 dt \\
  &\leq C^* \left[ \mathbb E \int\limits_0^T \| \Phi(t)\|_{\mathcal{L}_{(HS)}(Q^{1/2}(H);H)}^2 dt + \mathbb E \int\limits_0^T \|y(t;u,v)\|_H^2 dt + \int\limits_0^T \|y_d(t)\|_H^2 dt \right].
 \end{align*}
\end{proof}

\subsection{Approximation by a Strong Formulation}\label{sec:approximation}

In general, a duality principle of solutions to forward and backward SPDE's can be obtained by applying an Itô product formula.
This formula is not applicable to solutions in a mild sense.
Hence, we need to approximate the mild solutions of systems (\ref{frechet1}), (\ref{frechet2}) and (\ref{backwardstochstokes}) by strong formulations.
One method is given by introducing the Yosida approximation of the operator $A$, see \cite{seiid}.
For applications regarding duality principles, see \cite{smpf,euas}.
However, we apply the method introduced in \cite{yaos,soss}.
The basic idea is to formulate a mild solution with values in $D(A)$ by using the resolvent operator $R(\lambda)$ introduced in Section \ref{sec:functionalbackground}.
Thus, we get the required convergence results and the mild solutions coincide with the strong solutions.
In this section, we omit the dependence on the controls for the sake of simplicity.

\subsubsection{The Forward Equations}

Here, we provide approximations of the mild solutions to system (\ref{frechet1}) and system (\ref{frechet2}).
We introduce the following systems:

\begin{equation}\label{approxfrechet1}
 \left\{
 \begin{aligned}
  d z_1(t,\lambda) &= \left[- A z_1(t,\lambda) + R(\lambda) B u(t) \right] dt + R(\lambda) G( R(\lambda) z_1(t,\lambda)) d W(t), \\
  z_1(0,\lambda) &= 0,
 \end{aligned}
 \right.
\end{equation}
\begin{equation}\label{approxfrechet2}
 \left\{
 \begin{aligned}
  d z_2(t,\lambda) &= \left[- A z_2(t,\lambda) + A R(\lambda) D v(t) \right] dt + R(\lambda) G(R(\lambda) z_2(t,\lambda)) d W(t), \\
  z_2(0,\lambda) &= 0,
 \end{aligned}
 \right.
\end{equation}
where $\lambda > 0$, $u \in U$ and $v \in V$.
The process $(W(t))_{t \in [0,T]}$ is a Q-Wiener process with values in $H$ and covariance operator $Q \in \mathcal L(H)$.
The operators $A,R(\lambda),B,D,G$ and the spaces $U,V$ are introduced in Section \ref{sec:preliminaries} and Section \ref{sec:stochasticstokes}, respectively.

\begin{definition}
 a) A predictable process $(z_1(t,\lambda))_{t \in [0,T]}$ with values in $D(A)$ is called a \textbf{mild solution of system (\ref{frechet1})} if
 \begin{equation*}
  \mathbb E \int\limits_0^T \| z_1(t,\lambda)\|_{D(A)}^2 dt < \infty,
 \end{equation*}
 and we have for $t \in [0,T]$ and $\mathbb P$-a.s.
 \begin{equation*}
  z_1(t,\lambda) = \int\limits_0^t e^{-A (t-s)} R(\lambda) B u(s) ds + \int\limits_0^t e^{-A (t-s)} R(\lambda) G(R(\lambda) z_1(s,\lambda)) d W(s).
 \end{equation*}
 b) A predictable process $(z_2(t,\lambda))_{t \in [0,T]}$ with values in $D(A)$ is called a \textbf{mild solution of system (\ref{frechet2})} if
 \begin{equation*}
  \mathbb E \int\limits_0^T \| z_2(t,\lambda)\|_{D(A)}^2 dt < \infty,
 \end{equation*}
 and we have for $t \in [0,T]$ and $\mathbb P$-a.s.
 \begin{equation*}
  z_2(t,\lambda) = \int\limits_0^t e^{-A (t-s)} A R(\lambda) D v(s) ds + \int\limits_0^t e^{-A (t-s)}R(\lambda) G(R(\lambda) z_2(s,\lambda)) d W(s).
 \end{equation*}
\end{definition}

\begin{remark}
 Note that the approximation scheme provided in \cite{yaos,soss} differs to the approximation scheme introduced by system (\ref{approxfrechet1}) or system (\ref{approxfrechet2}).
 Here, the additional operator $R(\lambda)$ is necessary to obtain a duality principle. 
\end{remark}

Recall that the operators $R(\lambda)$ and $A R(\lambda)$ are linear and bounded on $H$.
Hence, existence and uniqueness results of mild solutions to system (\ref{approxfrechet1}) and system (\ref{approxfrechet2}) can be obtained similarly to Theorem \ref{existencestochstokes} for fixed $\lambda > 0$.
In the following lemma, we state that the mild solutions of system (\ref{approxfrechet1}) and system (\ref{approxfrechet2}) also satisfy a strong formulation, which is an immediate consequence of \cite[Proposition 2.3]{soss}.

\begin{lemma}\label{forwardstrong}
 Let $(z_1(t,\lambda))_{t \in [0,T]}$ and $(z_2(t,\lambda))_{t \in [0,T]}$ be the mild solutions of system (\ref{approxfrechet1}) and system (\ref{approxfrechet2}), respectively.
 Then we have for fixed $\lambda > 0$, $t \in [0,T]$ and $\mathbb P$-a.s.
 \begin{align*}
  z_1(t,\lambda) &= \int\limits_0^t (-A) z_1(s,\lambda) + R(\lambda) B u(s) ds + \int\limits_0^t R(\lambda) G(R(\lambda)z_1(s,\lambda)) d W(s), \\
  z_2(t,\lambda) &= \int\limits_0^t (-A) z_2(s,\lambda) + A R(\lambda) D v(s) ds + \int\limits_0^t R(\lambda) G(R(\lambda)z_2(s,\lambda)) d W(s).
 \end{align*}
\end{lemma}

We have the following convergence results.

\begin{lemma}\label{frechetconvergence}
 (i) Let $(z_1(t))_{t \in [0,T]}$ and $(z_1(t,\lambda))_{t \in [0,T]}$ be the mild solutions of system (\ref{frechet1}) and system (\ref{approxfrechet1}), respectively.
 Then we have
 \begin{equation*}
  \lim_{\lambda \rightarrow \infty} \mathbb E \int\limits_0^T \|z_1(t) - z_1(t,\lambda)\|_H^2 dt = 0.
 \end{equation*}
 (ii) Let $(z_2(t))_{t \in [0,T]}$ and $(z_2(t,\lambda))_{t \in [0,T]}$ be the mild solutions of system (\ref{frechet2}) and system (\ref{approxfrechet2}), respectively.
 Then we have
 \begin{equation*}
  \lim_{\lambda \rightarrow \infty} \mathbb E \int\limits_0^T \|z_2(t) - z_2(t,\lambda)\|_H^2 dt = 0.
 \end{equation*}
\end{lemma}

\begin{proof}
 First, we show part (i).
 Let $I$ be the identity operator in $H$.
 Recall that $G \colon  H \rightarrow \mathcal{L}_{(HS)}(Q^{1/2}(H);H)$ is linear and bounded.
 By definition, we have for all $\lambda > 0$, $t \in [0,T]$ and $\mathbb P$-a.s.
 \begin{align*}
  z_1(t) - z_1(t,\lambda) 
  &= \int\limits_0^t e^{-A (t-s)} [I-R(\lambda)] B u(s) ds + \int\limits_0^t e^{-A (t-s)}  G([I-R(\lambda)]z_1(s)) d W(s) \\
  &\quad + \int\limits_0^t e^{-A (t-s)}  [I-R(\lambda)] G(R(\lambda)z_1(s)) d W(s) + \int\limits_0^t e^{-A (t-s)} R(\lambda) G(R(\lambda)\left[z_1(s)-z_1(s,\lambda)\right]) d W(s).
 \end{align*}
 The remaining part of the proof can be obtained similarly to \cite[Lemma 3.1]{soss} using Lemma \ref{gronwall}.

 Next, we prove part (ii).
 By definition, we obtain for all $\lambda > 0$, $t \in [0,T]$ and $\mathbb P$-a.s.
 \begin{align*}
  z_2(t) - z_2(t,\lambda) 
  &= \int\limits_0^t A e^{-A (t-s)} [I-R(\lambda)] D v(s) ds + \int\limits_0^t e^{-A (t-s)}  G([I-R(\lambda)]z_2(s)) d W(s) \\
  &\quad + \int\limits_0^t e^{-A (t-s)}  [I-R(\lambda)] G(R(\lambda)z_2(s)) d W(s) + \int\limits_0^t e^{-A (t-s)} R(\lambda) G(R(\lambda)\left[z_2(s)-z_2(s,\lambda)\right]) d W(s).
 \end{align*}
 Thus, we get for all $\lambda > 0$ and $t \in [0,T]$
 \begin{equation}\label{ineq2}
  \mathbb E \, \left\| z_2(t)-z_2(t,\lambda) \right\|_H^2 \leq 4 \, \mathcal I_1(t,\lambda) + 4 \, \mathcal I_2(t,\lambda) + 4 \, \mathcal I_3(t,\lambda),
 \end{equation}
 where
 \begin{align*}
  \mathcal I_1(t,\lambda) &= \mathbb E \left\| \int\limits_0^t A e^{-A (t-s)} [I-R(\lambda)] D v(s) ds \right\|_H^2, \\
  \mathcal I_2(t,\lambda) &= \mathbb E \left\| \int\limits_0^t e^{-A (t-s)} G([I-R(\lambda)]z_2(s)) d W(s) \right\|_H^2 + \mathbb E \left\| \int\limits_0^t e^{-A (t-s)} [I-R(\lambda)] G(R(\lambda)z_2(s)) d W(s) \right\|_H^2 , \\
  \mathcal I_3(t,\lambda) &= \mathbb E \left\| \int\limits_0^t e^{-A (t-s)} R(\lambda) G(R(\lambda)\left[z_2(s)-z_2(s,\lambda)\right]) d W(s) \right\|_H^2.
 \end{align*}
 Recall that $D \colon V^0(\partial \mathcal D) \rightarrow D(A^\alpha)$ for all $\alpha \in \left(0,\frac{1}{4}\right)$.
 Using Lemma \ref{fractional}, equation (\ref{resolventcommutes}), Fubini's theorem and Young's inequality for convolutions, there exists a constant $C_1 > 0$ such that for all $\lambda > 0$ and all $t \in [0,T]$
 \begin{align}\label{ineq3}
  \int\limits_0^t \mathcal I_1(s,\lambda) \, ds 
  &\leq \mathbb E \int\limits_0^t \left(\int\limits_0^s \left\| A^{1-\alpha} e^{-A (s-r)} [I-R(\lambda)] A^\alpha D v(r) \right\|_H dr \right)^2 ds \nonumber \\
  &\leq C_1 \, \mathbb E \int\limits_0^T \left\| [I-R(\lambda)] A^\alpha D v(t) \right\|_H^2 dt.
 \end{align}
 Recall that $\left\| e^{-At} \right\|_{\mathcal L(H)} \leq 1$ for all $t \in [0,T]$.
 Due to the Itô isometry (\ref{isometry}) and Fubini's theorem, there exists a constant $C_2 > 0$ such that for all $\lambda > 0$ and all $t \in [0,T]$
 \begin{align}\label{ineq7}
  \int\limits_0^t \mathcal I_2(s,\lambda) \, ds 
  &\leq \int\limits_0^t \mathbb E \int\limits_0^s \left\| e^{-A (s-r)} G([I-R(\lambda)]z_2(r)) \right\|_{\mathcal{L}_{(HS)}(Q^{1/2}(H);H)}^2 dr \, ds \nonumber \\
  &\quad + \int\limits_0^t \mathbb E \int\limits_0^s \left\| e^{-A (s-r)} [I-R(\lambda)] G(R(\lambda)z_2(r)) \right\|_{\mathcal{L}_{(HS)}(Q^{1/2}(H);H)}^2 dr \, ds \nonumber \\
  &\leq C_2 \left[ \mathbb E \int\limits_0^T \left\| [I-R(\lambda)]z_2(t) \right\|_H^2 dt + \mathbb E \int\limits_0^T \left\| [I-R(\lambda)] G(R(\lambda)z_2(t)) \right\|_{\mathcal{L}_{(HS)}(Q^{1/2}(H);H)}^2 dt \right].
 \end{align}
 By the Itô isometry (\ref{isometry}), inequality (\ref{resolventinequality}) and Fubini's theorem, there exists a constant $C_3 > 0$ such that for all $\lambda > 0$ and all $t \in [0,T]$
 \begin{equation*}
  \mathcal I_3(t,\lambda) \leq C_3 \int\limits_0^t \mathbb E \, \left\| z_2(s)-z_2(s,\lambda) \right\|_H^2 ds.
 \end{equation*}
 Due to inequality (\ref{ineq2}), we get for all $\lambda > 0$ and $t \in [0,T]$
 \begin{equation*}
  \mathbb E \left\| z_2(t)-z_2(t,\lambda) \right\|_H^2 \leq 4 \; \mathcal I_1(t,\lambda) + 4 \; \mathcal I_2(t,\lambda) + 4 C_3 \int\limits_0^t \mathbb E \, \left\| z_2(s)-z_2(s,\lambda) \right\|_H^2 ds.
 \end{equation*}
 Applying Lemma \ref{gronwall}, we obtain for all $\lambda > 0$ and $t \in [0,T]$
 \begin{equation*}
  \mathbb E \left\| z_2(t)-z_2(t,\lambda) \right\|_H^2 \leq 4 \; \mathcal I_1(t,\lambda) + 4 \; \mathcal I_2(t,\lambda) + 16 C_3 e^{4 C_3 t} \int\limits_0^t [\mathcal I_1(s,\lambda) + \mathcal I_2(s,\lambda)] ds.
 \end{equation*}
 Using Fubini's theorem, inequality (\ref{ineq3}) and inequality (\ref{ineq7}), there exists a constant $C_3 > 0$ such that for all $\lambda > 0$
 \begin{align*}
  \mathbb E \int\limits_0^T \left\| z_2(t)-z_2(t,\lambda) \right\|_H^2 dt &\leq C^* \, \mathbb E \int\limits_0^T \left\| [I-R(\lambda)] A^\alpha D v(t) \right\|_H^2 dt + C^* \, \mathbb E \int\limits_0^T \left\| [I-R(\lambda)]z_2(t) \right\|_H^2 dt \\
  &\quad + C^* \, \mathbb E \int\limits_0^T \left\| [I-R(\lambda)] G(R(\lambda)z_2(t)) \right\|_{\mathcal{L}_{(HS)}(Q^{1/2}(H);H)}^2 dt.
 \end{align*}
 By equation (\ref{resolventconvergence}) and Lebesgue's dominated convergence theorem \cite[Theorem 2.8.1]{mt}, we can infer 
 \begin{equation*}
  \lim_{\lambda \rightarrow \infty} \mathbb E \int\limits_0^T \left\| z_2(t)-z_2(t,\lambda) \right\|_H^2 dt = 0.
 \end{equation*}
\end{proof}

\subsubsection{The Backward Equation}

Here we provide an approximation of the mild solution to system (\ref{backwardstochstokes}).
We introduce the following backward SPDE:
\begin{equation}\label{approxbackwardstochstokes}
 \left\{
 \begin{aligned}
  d z^*(t,\lambda) &= -[ -A z^*(t,\lambda) + R(\lambda) G^*(R(\lambda)\Phi(t,\lambda)) + R(\lambda) (y(t) - y_d(t))] dt + \Phi(t,\lambda) d W(t), \\
  z^*(T,\lambda) &= 0,
 \end{aligned}
 \right.
\end{equation}
where $\lambda > 0$. 
The process $(y(t))_{t \in [0,T]}$ is the mild solution of system (\ref{stochstokes}) and $(W(t))_{t \in [0,T]}$ is a Q-Wiener process with values in $H$ and covariance operator $Q \in \mathcal L(H)$.
The function $y_d \in L^2([0,T];H)$ is the desired velocity field.
The operators $A,R(\lambda),G^*$ are introduced in Section \ref{sec:functionalbackground} and Section \ref{sec:adjoint}, respectively.

\begin{definition}
 A pair of predictable processes $(z^*(t,\lambda),\Phi(t,\lambda))_{t \in [0,T]}$ with values in $D(A) \times \mathcal{L}_{(HS)}(Q^{1/2}(H);H)$ is called a \textbf{mild solution of system (\ref{approxbackwardstochstokes})} if
 \begin{align*}
  &\sup_{t \in [0,T]} \mathbb E \, \|z^*(t,\lambda)\|_{D(A)}^2 < \infty,
  &\mathbb E \int\limits_0^T \| \Phi(t,\lambda)\|_{\mathcal{L}_{(HS)}(Q^{1/2}(H);H)}^2 dt < \infty,
 \end{align*}
 and we have for all $t \in [0,T]$ and $\mathbb P$-a.s.
 \begin{equation*}
  z^*(t,\lambda) = \int\limits_t^T e^{-A (s-t)} R(\lambda) G^*(R(\lambda) \Phi(s,\lambda)) ds + \int\limits_t^T e^{-A (s-t)} R(\lambda) \left( y(s) - y_d(s)\right) ds - \int\limits_t^T e^{-A (s-t)} \Phi(s,\lambda) d W(s).
 \end{equation*}
\end{definition}

Recall that the operators $R(\lambda)$ and $A R(\lambda)$ are linear and bounded in $H$.
Hence, existence and uniqueness results of the mild solution to system (\ref{approxbackwardstochstokes}) can be obtained similarly to \cite{asoa}.
In the following lemma, we state that the mild solution of system (\ref{approxbackwardstochstokes}) also satisfies a strong formulation, which is an immediate consequence of \cite[Theorem 4.2]{smaw}.

\begin{lemma}\label{backwardstrong}
 Let the pair of stochastic processes $(z^*(t,\lambda),\Phi(t,\lambda))_{t \in [0,T]}$ be the mild solution of system (\ref{approxbackwardstochstokes}).
 Then we have for fixed $\lambda > 0$, all $t \in [0,T]$ and $\mathbb P$-a.s.
 \begin{equation*}
  z^*(t,\lambda) = \int\limits_t^T (-A) z^*(s,\lambda) + R(\lambda) G^*(R(\lambda)\Phi(s,\lambda)) + R(\lambda) \left( y(s) - y_d(s)\right) ds - \int\limits_t^T \Phi(s,\lambda) d W(s).
 \end{equation*}
\end{lemma}

We have the following convergence results.

\begin{lemma}\label{convergencebackward}
 Let $(z^*(t),\Phi(t))_{t \in [0,T]}$ and $(z^*(t,\lambda),\Phi(t,\lambda))_{t \in [0,T]}$ be the mild solutions of system (\ref{backwardstochstokes}) and system (\ref{approxbackwardstochstokes}), respectively.
 Then we have
 \begin{align*}
  &\lim_{\lambda \rightarrow \infty} \sup_{t \in [0,T]} \mathbb E \, \|z^*(t) - z^*(t,\lambda)\|_H^2 = 0,
  &\lim_{\lambda \rightarrow \infty} \mathbb E \int\limits_0^T \|\Phi(t) - \Phi(t,\lambda)\|_{\mathcal{L}_{(HS)}(Q^{1/2}(H);H)}^2 dt = 0.
 \end{align*}
\end{lemma}

\begin{proof}
 Let $I$ be the identity operator in $H$.
 By definition, we have for all $\lambda > 0$, all $t \in [0,T]$ and $\mathbb P$-a.s.
 \begin{align} \label{diff2}
  z^*(t) - z^*(t,\lambda) &= \int\limits_t^T e^{-A (s-t)} [G^*(\Phi(s))-R(\lambda) G^*(R(\lambda)\Phi(s,\lambda))] ds \nonumber \\
  &\quad + \int\limits_t^T e^{-A (s-t)} [I-R(\lambda)] \left( y(s) - y_d(s)\right) ds - \int\limits_t^T e^{-A (s-t)} [\Phi(s)-\Phi(s,\lambda)] d W(s).
 \end{align}
 Recall that the operator $G^* \colon \mathcal{L}_{(HS)}(Q^{1/2}(H);H) \rightarrow H$ is linear and bounded.
 Hence, we get for all $\lambda > 0$, all $t \in [0,T]$ and $\mathbb P$-a.s.
 \begin{align*}
  z^*(t) - z^*(t,\lambda) 
  &= \int\limits_t^T e^{-A (s-t)} G^*([I-R(\lambda)]\Phi(s)) ds + \int\limits_t^T e^{-A (s-t)} [I-R(\lambda)] G^*(R(\lambda)\Phi(s)) ds \nonumber \\
  &\quad + \int\limits_t^T e^{-A (s-t)} R(\lambda) G^*(R(\lambda)[\Phi(s)-\Phi(s,\lambda)]) ds + \int\limits_t^T e^{-A (s-t)} [I-R(\lambda)] \left(y(s) - y_d(s)\right) ds \nonumber \\
  &\quad - \int\limits_t^T e^{-A (s-t)} [\Phi(s)-\Phi(s,\lambda)] d W(s).
 \end{align*}
 Note that the assumptions of Lemma \ref{backwardlemma} are fulfilled.
 Thus, inequalities (\ref{backwardinequality1}) and (\ref{backwardinequality2}) hold.
 Let $T_1 \in [0,T)$.
 We obtain for all $\lambda > 0$
 \begin{align}
  &\sup_{t \in [T_1,T]} \mathbb E \, \|z^*(t) - z^*(t,\lambda)\|_H^2 \leq 4 c (T-T_1) \left[ \mathcal I_1(\lambda) + \mathcal I_2(\lambda) \right], \label{ineq8} \\
  &\mathbb E \int\limits_{T_1}^T \|\Phi(t) - \Phi(t,\lambda)\|_{\mathcal{L}_{(HS)}(Q^{1/2}(H);H)}^2 dt \leq 4 c (T-T_1) \left[ \mathcal I_1(\lambda) + \mathcal I_2(\lambda) \right], \label{ineq4}
 \end{align}
 where
 \begin{align*}
  &\mathcal I_1(\lambda) = \mathbb E \int\limits_{T_1}^T \left[ \| G^*([I-R(\lambda)]\Phi(t)) \|_H^2 + \| [I-R(\lambda)] G^*(R(\lambda)\Phi(t)) \|_H^2 + \| [I-R(\lambda)] \left(y(t) - y_d(t)\right) \|_H^2 \right] dt, \\
  &\mathcal I_2(\lambda) = \mathbb E \int\limits_{T_1}^T\| R(\lambda) G^*(R(\lambda)[\Phi(t)-\Phi(t,\lambda)]) \|_H^2 dt.
 \end{align*}
 Using equation (\ref{resolventconvergence}) and the Lebesgue's dominated convergence theorem \cite[Theorem 2.8.1]{mt}, we can conclude
 \begin{equation}\label{convergence1}
  \lim\limits_{\lambda \rightarrow \infty} \mathcal I_1(\lambda) = 0.
 \end{equation}
 By inequality (\ref{resolventinequality}), there exists a constant $C^*>0$ such that for all $\lambda > 0$
 \begin{equation}\label{ineq5}
  \mathcal I_2(\lambda) \leq C^* \, \mathbb E \int\limits_{T_1}^T \|\Phi(t) - \Phi(t,\lambda)\|_{\mathcal{L}_{(HS)}(Q^{1/2}(H);H)}^2 dt.
 \end{equation}
 Due to inequality (\ref{ineq4}) and inequality (\ref{ineq5}), we get for all $\lambda > 0$
 \begin{equation*}
  \mathbb E \int\limits_{T_1}^T \|\Phi(t) - \Phi(t,\lambda)\|_{\mathcal{L}_{(HS)}(Q^{1/2}(H);H)}^2 dt \leq 4 c (T-T_1) \, \mathcal I_1(\lambda) + 4 c \, C^* (T-T_1) \, \mathbb E \int\limits_{T_1}^T \|\Phi(t) - \Phi(t,\lambda)\|_{\mathcal{L}_{(HS)}(Q^{1/2}(H);H)}^2 dt.
 \end{equation*}
 We chose $T_1 \in [0,T)$ such that $4 c \, C^* (T-T_1) < 1$.
 Thus, we have for all $\lambda > 0$
 \begin{equation*}
  \mathbb E \int\limits_{T_1}^T \|\Phi(t) - \Phi(t,\lambda)\|_{\mathcal{L}_{(HS)}(Q^{1/2}(H);H)}^2 dt \leq \frac{4 c (T-T_1) \, \mathcal I_1(\lambda)}{1-4 c \, C^* (T-T_1)}.
 \end{equation*}
 Due to equation (\ref{convergence1}), we can conclude
 \begin{equation} \label{convergence2}
  \lim_{\lambda \rightarrow \infty} \mathbb E \int\limits_{T_1}^T \|\Phi(t) - \Phi(t,\lambda)\|_{\mathcal{L}_{(HS)}(Q^{1/2}(H);H)}^2 dt = 0.
 \end{equation}
 Using inequality (\ref{ineq8}), inequality (\ref{ineq5}), equation (\ref{convergence1}) and equation (\ref{convergence2}), we have
 \begin{equation*}
  \lim_{\lambda \rightarrow \infty} \sup_{t \in [T_1,T]} \mathbb E \, \|z^*(t) - z^*(t,\lambda)\|_H^2 = 0.
 \end{equation*}
 By equation (\ref{diff2}), we get for all $\lambda > 0$, all $t \in [0,T_1]$ and $\mathbb P$-a.s.
 \begin{align*} 
  z^*(t) - z^*(t,\lambda) &= e^{-A (T_1-t)}[z^*(T_1) - z^*(T_1,\lambda)] + \int\limits_t^{T_1} e^{-A (s-t)} [G^*(\Phi(s))-R(\lambda) G^*(R(\lambda)\Phi(s,\lambda))] ds \nonumber \\
  &\quad + \int\limits_t^{T_1} e^{-A (s-t)} [I-R(\lambda)] \left( y(s) - y_d(s)\right) ds - \int\limits_t^{T_1} e^{-A (s-t)} [\Phi(s)-\Phi(s,\lambda)] d W(s).
 \end{align*}
 Again, we find $T_2 \in [0,T_1]$ such that
 \begin{align*}
  &\lim_{\lambda \rightarrow \infty} \sup_{t \in [T_2,T_1]} \mathbb E \, \|z^*(t) - z^*(t,\lambda)\|_H^2 dt = 0,
  &\lim_{\lambda \rightarrow \infty} \mathbb E \int\limits_{T_2}^{T_1} \|\Phi(t) - \Phi(t,\lambda)\|_{\mathcal{L}_{(HS)}(Q^{1/2}(H);H)}^2 dt = 0.
 \end{align*}
 By continuing the method, we obtain the result.
\end{proof}

\section{Main Results}

\subsection{Duality Principle}

Based on the results provided in the previous sections, we are able to show a duality principle.
Since we formulated a control problem with simultaneous distributed controls and boundary controls, we obtain two equations.
The first equation gives us a relation between the mild solution of system (\ref{frechet1}) and the mild solution of the adjoint equation (\ref{backwardstochstokes}).
The second equation provides a relation between the mild solution of system (\ref{frechet2}) and the mild solution of the adjoint equation (\ref{backwardstochstokes}).

\begin{theorem}\label{duality}
 Let $(y(t;u,v))_{t \in [0,T]}$ and $(z^*(t;u,v),\Phi(t;u,v))_{t \in [0,T]}$ be the mild solutions of system (\ref{stochstokes}) and system (\ref{backwardstochstokes}) corresponding to the distributed control $u \in U$ and the boundary control $v \in V$, repsectively.
 Moreover, let $(z_1(t;\tilde u))_{t \in [0,T]}$ and $(z_2(t;\tilde v))_{t \in [0,T]}$ be the mild solutions of system (\ref{frechet1}) and system (\ref{frechet2}) corresponding to the controls $\tilde u \in U$ and $\tilde v \in V$, respectively.
 Then we have for all $\alpha \in (0,1/4)$
 \begin{align}
  &\mathbb E \int \limits_0^T \left\langle y(t;u,v)-y_d(t), z_1(t;\tilde u) \right\rangle_H dt = \mathbb E \int \limits_0^T \left\langle z^*(t;u,v), B \tilde u(t) \right\rangle_H dt, \label{dualityequation1} \\
  &\mathbb E \int \limits_0^T \left\langle y(t;u,v)-y_d(t), z_2(t;\tilde v) \right\rangle_H dt = \mathbb E \int \limits_0^T \left\langle A^{1-\alpha} z^*(t;u,v), A^\alpha D \tilde v(t) \right\rangle_H dt. \label{dualityequation2}
 \end{align}
\end{theorem}

\begin{proof}
 For the sake of simplicity, we omit the dependence on the controls.  
 First, we prove the result for the approximations derived in Section \ref{sec:approximation}.
 Let $(z_1(t,\lambda))_{t \in [0,T]}$ and $(z_2(t,\lambda))_{t \in [0,T]}$ be the mild solutions of system (\ref{approxfrechet1}) and system (\ref{approxfrechet2}), respectively.
 Using Lemma \ref{forwardstrong}, we have for all $\lambda > 0$, $t \in [0,T]$ and $\mathbb P$-a.s.
 \begin{align}
  z_1(t,\lambda) &= \int\limits_0^t (-A) z_1(s,\lambda) + R(\lambda) B \tilde u(s) ds + \int\limits_0^t R(\lambda) G(R(\lambda)z_1(s,\lambda)) d W(s), \label{strongrep1} \\
  z_2(t,\lambda) &= \int\limits_0^t (-A) z_2(s,\lambda) + A R(\lambda) D \tilde v(s) ds + \int\limits_0^t R(\lambda) G(R(\lambda)z_2(s,\lambda)) d W(s). \label{strongrep2}
 \end{align}
 Next, let the pair of stochastic processes $(z^*(t,\lambda),\Phi(t,\lambda))_{t \in [0,T]}$ be the mild solution of system (\ref{approxbackwardstochstokes}).
 Due to Lemma \ref{backwardstrong}, we get for all $\lambda > 0$, all $t \in [0,T]$ and $\mathbb P$-a.s.
 \begin{equation}\label{strongrep3}
  z^*(t,\lambda) = \int\limits_t^T (-A) z^*(s,\lambda) + R(\lambda) G^*(R(\lambda)\Phi(s,\lambda)) + R(\lambda) \left( y(s) - y_d(s)\right) ds - \int\limits_t^T \Phi(s,\lambda) d W(s).
 \end{equation}
 By definition, the process $(z^*(t,\lambda))_{t \in [0,T]}$ is predictable.
 Hence, we have for all $\lambda > 0$, all $t \in [0,T]$ and $\mathbb P$-a.s.
 \begin{align*}
  z^*(t,\lambda) &= \mathbb E \left[ \left. \int\limits_0^T (-A) z^*(s,\lambda) + R(\lambda) G^*(R(\lambda)\Phi(s,\lambda)) + R(\lambda) \left( y(s) - y_d(s)\right) ds \right| \mathcal F_t \right] \\
  &\quad - \int\limits_0^t (-A) z^*(s,\lambda) + R(\lambda) G^*(R(\lambda)\Phi(s,\lambda)) + R(\lambda) \left( y(s) - y_d(s)\right) ds.
 \end{align*}
 By the martingale representation theorem given by Proposition \ref{martingalerepresentation} with $(M(t))_{t \in [0,T]}$ satisfying for all $t \in [0,T]$ and $\mathbb P$-a.s. 
 \begin{equation*}
  M(t) = \mathbb E \left[ \left. \int\limits_0^T (-A) z^*(s,\lambda) + R(\lambda) G^*(R(\lambda)\Phi(s,\lambda)) + R(\lambda) \left( y(s) - y_d(s)\right) ds \right| \mathcal F_t \right],
 \end{equation*}
 there exists a unique predictable process $(\Psi(t,\lambda))_{t \in [0,T]}$ with values in $\mathcal{L}_{(HS)}(Q^{1/2}(H);H)$ such that for all $\lambda > 0$, all $t \in [0,T]$ and $\mathbb P$-a.s.
 \begin{align}\label{strongrep4}
  z^*(t,\lambda) &= \mathbb E \left[ \int\limits_0^T (-A) z^*(s,\lambda) + R(\lambda) G^*(R(\lambda)\Phi(s,\lambda)) + R(\lambda) \left( y(s) - y_d(s)\right) ds \right] \nonumber \\
  &\quad - \int\limits_0^t (-A) z^*(s,\lambda) + R(\lambda) G^*(R(\lambda)\Phi(s,\lambda)) + R(\lambda) \left( y(s) - y_d(s)\right) ds + \int\limits_0^t \Psi(s,\lambda) dW(s).
 \end{align}
 Since the pair $(z^*(t,\lambda),\Phi(t,\lambda))_{t \in [0,T]}$ satisfies equation (\ref{strongrep3}) uniquely, we can conclude $\Psi(t,\lambda) = \Phi(t,\lambda)$ for all $\lambda > 0$, almost all $t \in [0,T]$ and $\mathbb P$-almost surely.
 Applying the Itô product formula given by Lemma \ref{productformula} to equation (\ref{strongrep1}) and equation (\ref{strongrep4}), we get for all $\lambda > 0$, all $t \in [0,T]$ and $\mathbb P$-a.s.
 \begin{equation*}
  \left\langle z_1(t,\lambda),z^*(t,\lambda) \right\rangle_H = \mathcal I_1(t,\lambda) + \mathcal I_2(t,\lambda) + \mathcal I_3(t,\lambda) + \mathcal I_4(t,\lambda),
 \end{equation*}
 where
 \begin{align*}
  \mathcal I_1(t,\lambda) &= \int\limits_0^t \left[ \left\langle z_1(s,\lambda), A z^*(s,\lambda) \right\rangle_H - \left \langle z^*(s,\lambda) , A z_1(s,\lambda) \right\rangle_H \right] ds, \\
  \mathcal I_2(t,\lambda) &= \int\limits_0^t \left[\left\langle R(\lambda) G(R(\lambda)z_1(s,\lambda)), \Phi(s,\lambda) \right\rangle_{\mathcal{L}_{(HS)}(Q^{1/2}(H),H)} - \left \langle z_1(s,\lambda), R(\lambda) G^*(R(\lambda)\Phi(s,\lambda)) \right\rangle_H \right] ds, \\
  \mathcal I_3(t,\lambda) &= \int\limits_0^t \left \langle z^*(s,\lambda) , R(\lambda) B \tilde u(s) \right\rangle_H ds - \int\limits_0^t \left \langle z_1(s,\lambda), R(\lambda) \left( y(s) - y_d(s)\right) \right\rangle_H ds, \\
  \mathcal I_4(t,\lambda) &= \int\limits_0^t \left \langle z_1(s,\lambda) ,\Phi(s,\lambda) d W(s) \right\rangle_H + \int\limits_0^t \left \langle z^*(s,\lambda) , R(\lambda) G(R(\lambda)z_1(s,\lambda)) d W(s) \right\rangle_H.
 \end{align*}
 By definition, we have $z^*(T,\lambda) = 0$ for all $\lambda > 0$ and $\mathbb P$-almost surely.
 Hence, we obtain for all $\lambda > 0$ and $\mathbb P$-a.s.
 \begin{align}\label{eq2}
  0 &= \mathcal I_1(T,\lambda) + \mathcal I_2(T,\lambda) + \mathcal I_3(T,\lambda) + \mathcal I_4(T,\lambda).
 \end{align}
 Since the operator $A$ is self adjoint, we have for all $\lambda > 0$ and $\mathbb P$-a.s.
 \begin{equation}\label{eq3}
  \mathcal I_1(T,\lambda) = 0.
 \end{equation}
 Recall that the operator $R(\lambda)$ is self adjoint on $H$.
 Using equation (\ref{adjoint}), we obtain for all $\lambda > 0$ and $\mathbb P$-a.s.
 \begin{equation}\label{eq4}
  \mathcal I_2(T,\lambda) = 0.
 \end{equation}
 By equations (\ref{eq2}) -- (\ref{eq4}) and $\mathbb E \, \mathcal I_4(T,\lambda) = 0$ for all $\lambda > 0$, we get for all $\lambda > 0$
 \begin{equation*}
  0 = \mathbb E \, \mathcal I_3(T,\lambda).
 \end{equation*}
 Hence, we have for all $\lambda > 0$
 \begin{equation}\label{eq5}
  \mathbb E \int\limits_0^T \left \langle R(\lambda) z_1(t,\lambda), y(t) - y_d(t) \right\rangle_H dt = \mathbb E \int\limits_0^T \left \langle R(\lambda) z^*(t,\lambda) , B \tilde u(t) \right\rangle_H dt.
 \end{equation}
 Next, we show that the left hand side and the right hand side of equation (\ref{eq5}) converge as $\lambda \rightarrow \infty$.
 By the Cauchy-Schwarz inequality and inequality (\ref{resolventinequality}), we have for all $\lambda > 0$
 \begin{align*}
  &\left| \mathbb E \int\limits_0^T \left \langle z_1(t), y(t) - y_d(t) \right\rangle_H dt - \mathbb E \int\limits_0^T \left \langle R(\lambda) z_1(t,\lambda), y(t) - y_d(t) \right\rangle_H dt \right|^2 \\
  &\leq 2 \left| \mathbb E \int\limits_0^T \left \langle [I - R(\lambda)] z_1(t), y(t) - y_d(t) \right\rangle_H dt \right|^2 +  2 \left| \mathbb E \int\limits_0^T \left \langle R(\lambda) (z_1(t)-z_1(t,\lambda)), y(t) - y_d(t) \right\rangle_H dt \right|^2 \\
  &\leq 4 \left( \mathbb E \int\limits_0^T \left\| y(t) \right\|_H^2 dt + \int\limits_0^T \left\| y_d(t) \right\|_H^2 dt  \right) \left( \mathbb E \int\limits_0^T \left\| [I - R(\lambda)] z_1(t) \right\|_H^2 dt + \mathbb E \int\limits_0^T \left\| z_1(t)-z_1(t,\lambda) \right\|_H^2 dt\right).
 \end{align*}
 Using equation (\ref{resolventconvergence}), Lebesgue's dominated convergence theorem \cite[Theorem 2.8.1]{mt} and Lemma \ref{frechetconvergence}, we can conclude
 \begin{equation}\label{eq6}
  \lim\limits_{\lambda \rightarrow \infty} \mathbb E \int\limits_0^T \left \langle R(\lambda) z_1(t,\lambda), y(t) - y_d(t) \right\rangle_H dt = \mathbb E \int\limits_0^T \left \langle z_1(t), y(t) - y_d(t) \right\rangle_H dt.
 \end{equation}
 Recall that the operator $B \colon H \rightarrow H$ is bounded.
 Similarly as above, there exists a constant $C^*>0$ such that for all $\lambda > 0$
 \begin{align*}
  &\left| \mathbb E \int\limits_0^T \left \langle z^*(t) , B \tilde u(t) \right\rangle_H dt - \mathbb E \int\limits_0^T \left \langle R(\lambda) z^*(t,\lambda) , B \tilde u(t) \right\rangle_H dt\right|^2 \\
  &\leq 2 \left| \mathbb E \int\limits_0^T \left \langle [I-R(\lambda)] z^*(t) , B \tilde u(t) \right\rangle_H dt \right|^2 + 2 \left| \mathbb E \int\limits_0^T \left \langle R(\lambda) (z^*(t)- z^*(t,\lambda)) , B \tilde u(t) \right\rangle_H dt \right|^2 \\
  &\leq C^* \left( \mathbb E \int\limits_0^T \left\| \tilde u(t) \right\|_H^2 dt \right) \left( \mathbb E \int\limits_0^T \left\| [I - R(\lambda)] z^*(t) \right\|_H^2 dt + \sup_{t \in [0,T]} \mathbb E \, \left\| z^*(t)-z^*(t,\lambda) \right\|_H^2 \right).
 \end{align*}
 By equation (\ref{resolventconvergence}), Lebesgue's dominated convergence theorem \cite[Theorem 2.8.1]{mt} and Lemma \ref{convergencebackward}, we can infer
 \begin{equation*}
  \lim\limits_{\lambda \rightarrow \infty} \mathbb E \int\limits_0^T \left \langle R(\lambda) z^*(t,\lambda) , B \tilde u(t) \right\rangle_H dt = \mathbb E \int\limits_0^T \left \langle z^*(t) , B \tilde u(t) \right\rangle_H dt.
 \end{equation*}
 We conclude that the left hand side and the right hand side of equation (\ref{eq5}) converge as $\lambda \rightarrow \infty$ and equation (\ref{dualityequation1}) holds.
 
 Next, we show that equation (\ref{dualityequation2}) holds.
 Again, we apply Lemma \ref{productformula} to equation (\ref{strongrep2}) and equation (\ref{strongrep4}).
 Similarly to equation (\ref{eq5}), we find for all $\lambda > 0$ and all $\alpha \in (0,1/4)$
 \begin{equation}\label{eq7}
  \mathbb E \int \limits_0^T \left\langle R(\lambda) z_2(t,\lambda), y(t) - y_d(t) \right\rangle_H dt = \mathbb E \int \limits_0^T \left\langle R(\lambda) A^{1-\alpha} z^*(t,\lambda), A^\alpha D \tilde v(t) \right\rangle_H dt.
 \end{equation}
 Similarly to equation (\ref{eq6}), we can conclude
 \begin{equation*}
  \lim\limits_{\lambda \rightarrow \infty} \mathbb E \int\limits_0^T \left \langle R(\lambda) z_2(t,\lambda), y(t) - y_d(t) \right\rangle_H dt = \mathbb E \int\limits_0^T \left \langle z_2(t), y(t) - y_d(t) \right\rangle_H dt.
 \end{equation*}
 Recall that the operator $A^\alpha D \colon V^0(\partial \mathcal D) \rightarrow H$ is bounded for all $\alpha \in (0,1/4)$.
 Hence, the process $(A^\alpha D \tilde v(t))_{t \in [0,T]}$ takes values in $H$ such that $\mathbb E \int_0^T \|A^\alpha D \tilde v(t)\|_H^2 dt < \infty$.
 Since $D(A^{1-\alpha})$ is dense in $H$, there exists a sequence of processes $(v_m(t))_{t \in [0,T]}$, $m \in \mathbb N$, taking values in $D(A^{1-\alpha})$ such that $\mathbb E \int_0^T \|v_m(t)\|_{D(A^{1-\alpha})}^2 dt < \infty$ for each $m \in \mathbb N$ and
 \begin{equation*}
  \lim_{m \rightarrow \infty} \mathbb E \int\limits_0^T \|A^\alpha D \tilde v(t) - v_m(t)\|_H^2 dt = 0.
 \end{equation*}
 Due to Proposition \ref{backwardregularity}, the process $(z^*(t))_{t \in [0,T]}$ takes values in $D(A^{1-\alpha})$ for all $\alpha \in (0,1/4)$.
 By equation (\ref{resolventcommutes}), Lemma \ref{fractionalselfadjoint}, the Cauchy-Schwarz inequality, inequality (\ref{resolventinequality}) and Fubini's theorem, there exists a constant $C^*>0$ such that for all $\lambda > 0$, all $\alpha \in (0,1/4)$ and each $m \in \mathbb N$
 \begin{align*}
  &\left| \mathbb E \int \limits_0^T \left\langle A^{1-\alpha} z^*(t), v_m(t) \right\rangle_H dt - \mathbb E \int \limits_0^T \left\langle R(\lambda) A^{1-\alpha} z^*(t,\lambda), v_m(t) \right\rangle_H dt\right|^2 \\
  &\leq 2 \left| \mathbb E \int\limits_0^T \left \langle [I-R(\lambda)] z^*(t) , A^{1-\alpha} v_m(t) \right\rangle_H dt \right|^2 + 2 \left| \mathbb E \int\limits_0^T \left \langle R(\lambda) (z^*(t)- z^*(t,\lambda)) , A^{1-\alpha} v_m(t) \right\rangle_H dt \right|^2 \\
  &\leq C^* \left( \mathbb E \int\limits_0^T \left\| v_m(t) \right\|_{D(A^{1-\alpha})}^2 dt \right) \left( \mathbb E \int\limits_0^T \left\| [I - R(\lambda)] z^*(t) \right\|_H^2 dt + \sup_{t \in [0,T]} \mathbb E \, \left\| z^*(t)-z^*(t,\lambda) \right\|_H^2 \right).
 \end{align*}
 Using equation (\ref{resolventconvergence}), Lebesgue's dominated convergence theorem \cite[Theorem 2.8.1]{mt} and Lemma \ref{convergencebackward}, we can infer for each $m \in \mathbb N$
 \begin{equation*}
  \lim\limits_{\lambda \rightarrow \infty} \mathbb E \int \limits_0^T \left\langle R(\lambda) A^{1-\alpha} z^*(t,\lambda), v_m(t) \right\rangle_H dt = \mathbb E \int \limits_0^T \left\langle A^{1-\alpha} z^*(t), v_m(t) \right\rangle_H dt.
 \end{equation*}
 Due to the Moore-Osgood theorem \cite[Theorem 7.11]{poma}, we get
 \begin{align*}
  \lim\limits_{\lambda \rightarrow \infty} \mathbb E \int \limits_0^T \left\langle R(\lambda) A^{1-\alpha} z^*(t,\lambda), A^\alpha D \tilde v(t) \right\rangle_H dt 
  &= \lim\limits_{\lambda \rightarrow \infty} \lim\limits_{m \rightarrow \infty}\mathbb E \int \limits_0^T \left\langle R(\lambda) A^{1-\alpha} z^*(t,\lambda), v_m(t) \right\rangle_H dt \\
  &= \lim\limits_{m \rightarrow \infty} \lim\limits_{\lambda \rightarrow \infty} \mathbb E \int \limits_0^T \left\langle R(\lambda) A^{1-\alpha} z^*(t,\lambda), v_m(t) \right\rangle_H dt \\
  &= \mathbb E \int \limits_0^T \left\langle A^{1-\alpha} z^*(t), A^\alpha D \tilde v(t) \right\rangle_H dt.
 \end{align*}
 We conclude that the left hand side and the right hand side of equation (\ref{eq7}) converge as $\lambda \rightarrow \infty$ and equation (\ref{dualityequation2}) holds.
\end{proof}

\subsection{The Optimal Controls}

Based on the optimality conditions given by equation (\ref{optcondition1}) and equation (\ref{optcondition2}), we deduce formulas of the optimal controls using the duality principle derived in the previous theorem.

\begin{theorem}
 Let $(z^*(t;u,v),\Phi(t;u,v))_{t \in [0,T]}$ be the mild solution of system (\ref{backwardstochstokes}) corresponding to the controls $u \in U$ and $v \in V$.
 Then the optimal controls $\overline u \in U$ and $\overline v \in V$ satisfy for all $\alpha \in (0,\frac{1}{4})$, almost all $t \in [0,T]$ and $\mathbb P$-a.s.
 \begin{align}
  &\overline u(t) = - \frac{1}{\kappa_1} \, B^* z^*(t;\overline u,\overline v) , \label{optimalcontrol1} \\
  &\overline v(t) = - \frac{1}{\kappa_2} \, K^* A^{1-\alpha} z^*(t;\overline u,\overline v), \label{optimalcontrol2}
 \end{align}
 where $B^* \in \mathcal L(H)$ and $K^* \in \mathcal L(H;V^0(\partial \mathcal D))$ are the adjoint operators of $B \in \mathcal L(H)$ and $K = A^\alpha D \in \mathcal L(V^0(\partial \mathcal D);H)$, respectively.
\end{theorem}

\begin{proof}
 Let $(y(t;u,v))_{t \in [0,T]}$ and $(z_1(t;u))_{t \in [0,T]}$ be the mild solutions of system (\ref{stochstokes}) and system (\ref{frechet1}) corresponding to the controls $u \in U$ and $v \in V$, respectively.
 Using equation (\ref{optcondition1}) and Theorem \ref{costfunctionalfrechet}, the optimal control $\overline u \in U$ satisfies for every $\tilde u \in U$
 \begin{equation*}
  \mathbb E \int\limits_0^T \left\langle y(t;\overline u,\overline v) -y_d(t), z_1(t;\tilde u) \right\rangle_H dt + \kappa_1\, \mathbb E \int\limits_0^T \left\langle \overline u(t), \tilde u(t) \right\rangle_H dt = 0.
 \end{equation*}
 By Theorem \ref{duality}, we obtain for every $\tilde u \in U$
 \begin{equation*}
  \mathbb E \int \limits_0^T \left\langle z^*(t;\overline u,\overline v), B \tilde u(t) \right\rangle_H dt + \kappa_1\, \mathbb E \int\limits_0^T \left\langle \overline u(t), \tilde u(t) \right\rangle_H dt = 0.
 \end{equation*}
 Hence, we get for every $\tilde u \in U$
 \begin{equation*}
  \mathbb E \int \limits_0^T \left\langle B^* z^*(t;\overline u,\overline v) + \kappa_1\, \overline u(t), \tilde u(t) \right\rangle_H dt = 0.
 \end{equation*}
 Therefore, the optimal control $\overline u \in U$ satisfies equation (\ref{optimalcontrol1}) for almost all $t \in [0,T]$ and $\mathbb P$-almost surely.
 
 Let $(z_2(t;v))_{t \in [0,T]}$ be the mild solution of system (\ref{frechet2}) corresponding to the control $v \in V$.
 Due to equation (\ref{optcondition2}) and Theorem \ref{costfunctionalfrechet}, the optimal control $\overline v \in V$ fulfills the following equation for every $\tilde v \in V$:
 \begin{equation*}
  \mathbb E \int\limits_0^T \left\langle y(t;\overline u,\overline v) -y_d(t), z_2(t;\tilde v) \right\rangle_H dt + \kappa_2\, \mathbb E \int\limits_0^T \left\langle \overline v(t), \tilde v(t) \right\rangle_{V^0(\partial \mathcal D)} dt = 0.
 \end{equation*}
 By Theorem \ref{duality}, we have for all $\alpha \in (0,\frac{1}{4})$ and every $\tilde v \in V$ 
 \begin{equation*}
  \mathbb E \int \limits_0^T \left\langle A^{1-\alpha} z^*(t;\overline u,\overline v), A^\alpha D \tilde v(t) \right\rangle_H dt + \kappa_2\, \mathbb E \int\limits_0^T \left\langle \overline v(t), \tilde v(t) \right\rangle_{V^0(\partial \mathcal D)} dt = 0.
 \end{equation*}
 Hence, we get for all $\alpha \in (0,\frac{1}{4})$ and every $\tilde v \in V$
 \begin{equation*}
  \mathbb E \int \limits_0^T \left\langle K^* A^{1-\alpha} z^*(t;\overline u,\overline v) + \kappa_2\, \overline v(t), \tilde v(t) \right\rangle_{V^0(\partial \mathcal D)} dt = 0.
 \end{equation*}
 Therefore, the optimal control $\overline v \in V$ satisfies equation (\ref{optimalcontrol2}) for all $\alpha \in (0,\frac{1}{4})$, almost all $t \in [0,T]$ and $\mathbb P$-almost surely.
\end{proof}

\begin{remark}
 Let us denote by $(\overline y(t))_{t \in [0,T]}$ and $(\overline z^*(t),\overline \Phi(t))_{t \in [0,T]}$ the mild solutions of system (\ref{stochstokes}) and system (\ref{backwardstochstokes}) corresponding to the optimal controls $\overline u \in U$ and $\overline v \in V$, respectively.
 As a consequence of the previous theorem, the optimal controls can be computed by solving the stochastic boundary value problem imposed by the following system of coupled forward-backward SPDEs:
 \begin{equation}\label{forwardbackwardspde}
  \left\{
  \begin{aligned}
   d \, \overline y(t) &= \left[- A \overline y(t) - \frac{1}{\kappa_1} \, B B^* \overline z^*(t) - \frac{1}{\kappa_2} \, A D K^* A^{1-\alpha} \overline z^*(t) \right] dt + G(\overline y(t)) d W(t), \\
   d \, \overline z^*(t) &= -\left[ -A \overline z^*(t) + G^*\left(\overline \Phi(t)\right) + \overline y(t) - y_d(t)\right] dt + \overline \Phi(t) d W(t), \\
   \overline y(0) &= \xi, \quad \overline z^*(T) =0.
  \end{aligned} 
  \right.
 \end{equation}
 As a next step, computational methods for solving system (\ref{forwardbackwardspde}) need to be developed.
\end{remark}

\section{Conclusion}

In this paper, we considered a control problem constrained by the stochastic Stokes equations on connected and bounded domains with linear multiplicative noise, where controls are defined inside the domain as well as on the boundary.

We proved an existence and uniqueness result for the mild solution of the stochastic Stokes equations dependent on inhomogeneous tangential boundary conditions.
Based on the Fr\'echet derivative of the cost functional, we stated necessary and sufficient optimality conditions the optimal distributed control as well as the optimal boundary control have to satisfy.
Using the adjoint equation given by a backward SPDE, a duality principle was derived such that we deduced explicit formulas for the optimal controls.
As a consequence, the optimal velocity field can be obtained by solving a system of coupled forward-backward SPDEs.

For engineering applications to control problems of fluid dynamics, the inflow is often used as a boundary control, see \cite{sans} and the references therein.
These boundary controls can not be covered by tangential boundary conditions and thus remain as an open problem.

\section*{Acknowledgement}

This research is supported by a research grant of the 'International Max Planck Research School (IMPRS) for Advanced Methods in Process and System Engineering', Magdeburg.
The authors would like to thank Prof. Wilfried Grecksch of the Martin Luther University Halle-Wittenberg for his helpful advice on various technical issues.

\bibliography{references_stokes}
\bibliographystyle{plain}

\end{document}